\documentclass[a4paper,12pt,reqno]{amsart}

\usepackage{color}
\usepackage{amsmath}
\usepackage{amssymb}
\usepackage{amsfonts}
\usepackage{graphicx}
\usepackage{mathtools}
\usepackage[colorlinks]{hyperref}
\renewcommand\eqref[1]{(\ref{#1})} 
\usepackage{booktabs}
\usepackage{lscape}
\usepackage{xcolor}
\usepackage{tikz}
\usepackage{tikz-cd}
\usepackage{ circuitikz }

\usepackage{mathrsfs}
\usepackage{comment}
\usepackage[thinlines]{easytable}

\graphicspath{ {images/} }
\newcommand{\Q}{\mathbb{Q}}
\newcommand{\R}{\mathbb{R}}

\title[Integro-differential equations on locally compact groups]{Evolutionary integro-differential equations of scalar type on locally compact groups}

\author[S. G\'omez Cobos]{Santiago G\'omez Cobos}
\address{
	Santiago G\'omez Cobos:
	\endgraf
	Department of Mathematics: Analysis, Logic and Discrete Mathematics
	\endgraf
	Ghent University, Krijgslaan 281, Building S8, B 9000 Ghent
	\endgraf
	Belgium
	\endgraf
	{\it E-mail address} {\rm davidsantiago.gomezcobos@ugent.be}}

\author[J. E. Restrepo]{Joel E. Restrepo}
\address{
	Joel E. Restrepo:
	\endgraf
	Department of Mathematics: Analysis, Logic and Discrete Mathematics
	\endgraf
	Ghent University, Krijgslaan 281, Building S8, B 9000 Ghent
	\endgraf
	Belgium
	\endgraf
and
\endgraf
Department of Mathematics
\endgraf
Cinvestav IPN, Mexico City 
\endgraf
Mexico
\endgraf
	{\it E-mail address} {\rm joel.restrepo@cinvestav.mx; cocojoel89@yahoo.es}}

\author[M. Ruzhansky]{Michael Ruzhansky}
\address{
	Michael Ruzhansky:
	\endgraf
	Department of Mathematics: Analysis, Logic and Discrete Mathematics
	\endgraf
	Ghent University, Krijgslaan 281, Building S8, B 9000 Ghent
	\endgraf
	Belgium
	\endgraf
	and
	\endgraf
	School of Mathematical Sciences
		\endgraf Queen Mary University of London 
			\endgraf
		United Kingdom
			\endgraf
	{\it E-mail address} {\rm michael.ruzhansky@ugent.be}}


\newtheoremstyle{theorem}
{10pt}          
{10pt}  
{\sl}  
{\parindent}     
{\bf}  
{. }    
{ }    
{}     
\theoremstyle{theorem}

\numberwithin{equation}{section}
\theoremstyle{plain}
\newtheorem{thm}{Theorem}[section]
\newtheorem{prop}[thm]{Proposition}
\newtheorem{cor}[thm]{Corollary}
\newtheorem{lem}[thm]{Lemma}

\theoremstyle{definition}
\newtheorem{defn}[thm]{Definition}
\newtheorem{rem}[thm]{Remark}
\newtheorem{ex}[thm]{Example}

\newtheoremstyle{defi}
{10pt}          
{10pt}  
{\rm}  
{\parindent}     
{\bf}  
{. }    
{ }    
{}     
\theoremstyle{defi}

\begin{document}
 	\begin{abstract} 
We study existence, uniqueness, norm estimates and asymptotic time behaviour (in some cases can be claimed to be sharp) for the solution of a general evolutionary integral (differential) equation of scalar type on a locally compact separable unimodular group $G$ governed by any positive left invariant operator (unbounded and either with discrete or continuous spectrum) on $G$. We complement our studies by proving some time-space (Strichartz type) estimates for the classical heat equation, and its time-fractional counterpart, as well as for the time-fractional wave equation. The latter estimates allow us to give some results about the well-posedness of nonlinear partial integro-differential equations. We provide many examples of the results by considering particular equations, operators and groups through the whole paper. 
\end{abstract}
	\maketitle
            
            \tableofcontents

\section{Introduction}

Locally compact groups and their theory have been developed through the years by many researchers who addressed different problems about properties, structure, and many other issues; see e.g. the books \cite{intro-nuevo-3,intro-nuevo1,intro-nuevo-2,intro-nuevo-4}. Also, we mention the following papers \cite{intro-p1,intro-nuevo11,intro-nuevo1} and the references therein. We also refer to several papers on general functional and difference equations on locally compact groups \cite{intro-p4,intro-p3,intro-p2}. The generality and lack of the smooth structure on such groups have made difficult the studies of partial differential equations in this setting. Nevertheless, Fourier multipliers of Hörmander type on a locally compact group have been studied in \cite{RR2020}. From the point of view of partial differential equations, authors gave several applications by obtaining regularity properties of linear evolution equations. This was
done through an application of a spectral multiplier theorem, which was obtained by
using von Neumann algebras theory and the generalized non-commutative Lorentz spaces. By these results, one can translate the regularity problem to the study of boundedness of propagators
of the form $\varphi(|\mathscr{L}|)$, where we denote for the operator $\mathscr{L}$ by $\mathscr{L}=U|\mathscr{L}|$ its polar decomposition, for $\varphi$ a positive monotonically decreasing vanishing at infinity continuous function. If we can provide good asymptotic behaviour of the traces of the spectral projections $E_{(s,+\infty)}(|\mathscr{L}|)$, the problem changes to that of imposing good properties on the function $\varphi$.

In \cite{RR2020}, one considered the setting of semifinite von Neumann algebras. This assumption is very general since it allows us to consider simultaneously some type I and II groups. Hence, it is not necessary to restrict ourselves to any of those two categories. So, the type of locally compact groups that we can consider here is very large. For instance, compact, semi simple, exponential, nilpotent, some solvable ones, real algebraic, and many more. In particular, let us mention some well-known ones: the Euclidean space $\mathbb{R}^n$, the Heisenberg group $\mathbb{H}^n$, any graded Lie group, Engel group, Cartan group, any connected Lie group and the group of $p$-adic numbers.  

In this paper, we complement our previous studies by considering more general linear and nonlinear equations, and improving (extending) some of the results from \cite{RR2020,SRR,cras}. Let us now describe the setting and main ideas of the paper. The regularity is mainly associated with the study of Fourier (H\"ormander type) and spectral multipliers on a locally compact separable unimodular group. Briefly, let us explain it. In general, some of the new results of Fourier multipliers \cite{RR2020,SRR} can be applied to obtain results on spectral multipliers. The application of these latter results will help to estimate some norms for the solution of an evolutionary integral equation of scalar type (see equation \eqref{volterra-e}) on a locally compact separable unimodular group. The norms can be reduced to the estimation of its propagator (time dependence) in the noncommutative Lorentz space norm \cite{[51]} (the space is associated with a semifinite von Neumann algebra \cite{von,[46],von2}), that involves calculating the trace of the spectral projections of the operator $\mathscr{L}$ (\cite{pacific}). The considered operator (unbounded) can be any positive linear left invariant operator acting on $G$ with the possibility (generality) of having either continuous or discrete spectrum. Here and thereafter, we set up by convention that a positive operator means self-adjoint and nonnegative. Therefore, the operators which in principle can be involved are e.g.: Laplacians, sub-Laplacians, Rockland operators, subcoercive positive operators, and many more. Moreover, we are able to establish asymptotic time behavior for these equations that will depend on the considered kernel. An important remark is that at the first stage in \cite{RR2020}, it was only possible to apply those results for certain type of equations like the heat equation. Nevertheless, by some additional studies and the current investigation, we have enlarged the classes of equations that can be considered, e.g. heat and wave types, Schrödinger type, multi-term heat type, abstract Cauchy problem with variable coefficient, Rayleigh–Stokes type equations, kernels associated with Bernstein functions, etc. The first result in this direction was given in \cite{SRR}.

Now we provide the path and route with the statements of our problems. This will give us a better clarity of our present research. Hence, we first begin by mentioning the $L^p-L^q$ boundedness result for an operator $\mathscr{L}$ that will be applied for more general operators of the form $\varphi(\mathscr{L})$ for certain functions $\varphi.$ Full details about notations and objects can be found in the preliminaries (Section \ref{preli}). 

\begin{thm}\cite[Theorem 5.1]{RR2020}
Let G be a locally compact unimodular group. Let $1<p\leqslant 2\leqslant q<+\infty$ and assume that $\mathscr{L}$ is a left-invariant linear continuous operator on the Schwartz-Bruhat space $S(G)$. Then we have
\[
\|\mathscr{L}\|_{L^p(G)\to L^q(G)}\lesssim \|\mathscr{L}\|_{L^{r,\infty}(VN_R(G))}, 
\]
where $L^{r,\infty}$ is the noncommutative Lorentz spaces with $\frac{1}{r}=\frac{1}{p}-\frac{1}{q}$ $(p\neq q)$. While for $p=q=2$, we have the following sharp estimate
\[
\|\mathscr{L}\|_{L^2(G)\to L^2(G)}=\sup_{s\in\mathbb{R}^+}\mu_{s}(\mathscr{L}).
\]
\end{thm}
Having in mind that the boundedness of the right hand side of the above estimate depends intrinsically of the noncommutative Lorentz space \cite{[51]}, in \cite{RR2020}, it was also proved in a semifinite von Neuman algebra $M$ that:

\begin{thm}\cite[Theorem 6.1]{RR2020}\label{intro1}
    Let $\mathscr{L}$ be a closed (maybe unbounded) operator affiliated with a semifinite von Neuman algebra $M\subset \mathcal{B}(\mathcal{H})$. Let $\varphi$ be a monotonically decreasing continuous function on $[0,+\infty)$ such that $\varphi(0)=1$ and $\displaystyle\lim_{v\to+\infty}\varphi(v)=0.$ Then for every $1\leqslant r<+\infty$ we have the inequality 
\[
\|\varphi(|\mathscr{L}|)\|_{L^{r, \infty}(M)}=\sup_{v>0}\varphi(v)\big[\tau(E_{(0,v)}(|\mathscr{L}|))\big]^{\frac{1}{r}}.
\] 
\end{thm}
The latter result was targeted at providing $L^p-L^q$ boundedness estimates for the heat kernel as well as more general propagators of the form $\varphi(\mathscr{L})$. In the particular case of $M=VN_R(G)$ being the right von Neumann algebra of a locally compact unimodular group $G$, the following result is given:

\begin{thm}\label{intro2}
Let $G$ be a locally compact unimodular separable group and let $\mathscr{L}$ be a left invariant on $G.$ Let $\varphi$ be a monotonically decreasing continuous function on $[0,+\infty)$ such that $\varphi(0)=1$ and $\displaystyle\lim_{v\to+\infty}\varphi(v)=0.$ Then we obtain
\[
\|\varphi(|\mathscr{L}|)\|_{L^p(G)\to L^q(G)}\lesssim \sup_{v>0}\varphi(v)\big[\tau(E_{(0,v)}(|\mathscr{L}|))\big]^{\frac{1}{p}-\frac{1}{q}},\quad 1<p\leqslant 2\leqslant q<+\infty. 
\]
\end{thm}
Here it is important to mention that the operators $\mathscr{L}$ affiliated with the semifinite von Neumann algebra $VN_R(G)$ are those that are left invatiant on the group $G$ \cite[Remark 2.17]{RR2020}. 
Note that this last result allows us to think about some applications in partial differential equations. The first attempt to this idea was given in \cite[Section 7]{RR2020} for the heat equation. In fact, they consider the following equation:
\begin{equation}\label{intro-heat-e}
\partial_t w+\mathscr{L}w=0,\quad w(0)=w_0,
\end{equation}
where $\mathscr{L}$ is a positive unbounded operator affilated with $VN_R(G).$ For each $t>0,$ the solution (by using functional calculus \cite{BorelFunctional}) is given by $w(t,x)=e^{-t\mathscr{L}}w_0(x).$ This solution satisfies the equation and the initial condition. Hence, the $L^p-L^q$ boundedness of the solution follows as 
\[
\|w(t,\cdot)\|_{L^q(G)}=\|e^{-t\mathscr{L}}w_0\|_{L^q(G)}\lesssim \|e^{-t\mathscr{L}}\|_{L^{r,\infty}(VN_R(G))}\|w_0\|_{L^p(G)}.
\]
Imposing that 
\begin{equation}\label{intro-trace}
\tau\big(E_{(0,s)}(\mathscr{L})\big)\lesssim s^{\lambda},\quad s\to+\infty,\quad\text{for some}\quad \lambda>0,
\end{equation}
and by Theorem \ref{intro2}, we can get the following asymptotic behaviour of the heat propagator 
\[
\|w(t,\cdot)\|_{L^q(G)}\lesssim \|e^{-t\mathscr{L}}\|_{L^{r,\infty}(VN_R(G))}\|w_0\|_{L^p(G)}\lesssim t^{-\lambda\left(\frac{1}{p}-\frac{1}{q}\right)}\|w_0\|_{L^p(G)},\quad t>0.
\]
From the point of view of equations, we can see that Theorem \ref{intro2} is conditioned to the monotonicity and decay of the function $\varphi.$ Thus, e.g., we can not apply these results to oscillatory propagators among many others. In this regard, recently in \cite{SRR}, the authors of this manuscript have proven an important and interesting estimation for general propagators
in the noncommutative Lorentz space (associated to a semifinite von Neumann algebra). Indeed, we proved that the noncommutative Lorentz norm of a propagator of the form $\varphi(|\mathscr{L}|)$ can be estimated if the Borel function $\varphi$ is bounded by a positive monotonically decreasing vanishing at infinity continuous function $\psi$. This result, in particular, but not restricted, gives us the possibility to calculate $L^p-L^q$ $(1<p\leqslant 2\leqslant q<+\infty)$ norm estimates for the solutions of heat, wave and Schrödinger type
equations, here we mean the time-fractional versions of such equations (new at that time in this setting), on a locally compact separable unimodular group $G$ in terms of a non-local integro-differential operator in time and any positive left invariant operator (unbounded
and either with discrete or continuous spectrum) on $G$. In some cases, we can also obtain asymptotic
(in time) estimates for the solutions and even claim the sharpness of this decay. The obtained result was stated as follows:
\begin{thm}\cite[Theorem 1.1]{SRR}
    Let $\mathscr{L}$ be a closed (maybe unbounded) operator affiliated with a semifinite von Neuman algebra $M$. Let $\varphi$ be a Borel measurable function on $[0,+\infty)$. Suppose also that $\psi$ is a monotonically decreasing continuous function on $[0,+\infty)$ such that $\psi(0)=1$, $\displaystyle\lim_{v\to+\infty}\psi(v)=0$ and $|\varphi(v)|\leqslant \psi(v)$ for all $v\in[0,+\infty).$ Then for every $1\leqslant r<+\infty$ we have the inequality 
\[
\|\varphi(|\mathscr{L}|)\|_{L^{r, \infty}(M)}\leqslant \sup_{v>0}\psi(v)\big[\tau(E_{(0,v)}(|\mathscr{L}|))\big]^{\frac{1}{r}}.
\] 
\end{thm}
The above statement implies, in particular, the following result:
\begin{thm}\label{intro-lp-lq}
Let $G$ be a locally compact unimodular separable group and let $\mathscr{L}$ be a left invariant on $G.$ Let $\varphi$ be a Borel measurable function on $[0,+\infty)$. Suppose also that $\psi$ is a monotonically decreasing continuous function on $[0,+\infty)$ such that $\psi(0)=1$, $\displaystyle\lim_{v\to+\infty}\psi(v)=0$ and $|\varphi(v)|\leqslant \psi(v)$ for all $v\in[0,+\infty).$ Then we have 
\[
\|\varphi(|\mathscr{L}|)\|_{L^p(G)\to L^q(G)}\lesssim \sup_{v>0}\psi(v)\big[\tau(E_{(0,v)}(|\mathscr{L}|))\big]^{\frac{1}{p}-\frac{1}{q}},\quad 1<p\leqslant 2\leqslant q<+\infty. 
\]
\end{thm}
The result above gave us the possibility, at the first glance, to treat the following integro-differential equations:
\begin{itemize}
\item $\mathscr{L}$-heat type equations:
\begin{align*}
^{C}\partial_{t}^{\beta}w(s,x)-\mathscr{L}w(t,x)&=0,\quad t>0,\quad x\in G,\quad 0<\beta\leqslant 1, \\
w(t,x)|_{_{_{t=0}}}&=w_0(x), 
\end{align*}
where $\mathscr{L}$ is a positive left invariant operator in $G$ and the non-local operator (in time) $^{C}\partial_{t}^{\beta}w(t,x)$, it is the so-called Djrbashian--Caputo fractional derivative, defined by $^{C}\partial_{t}^{\beta}w(t,x)=\prescript{RL}{0}I^{1-\beta}\partial_t w(t,x),$ and $\prescript{RL}{0}I^{\rho}w(t,x)=\frac1{\Gamma(\rho)}\int_0^t (t-s)^{\rho-1}w(s,x)\,\mathrm{d}s,$ is the Riemann-Liouville fractional integral of order $\rho>0.$ 
\item $\mathscr{L}$-wave type equations:
\begin{align*}
^{C}\partial_{t}^{\beta}w(s,x)-\mathscr{L}w(t,x)&=0,\quad t>0,\quad x\in G,\quad 1<\beta<2, \\
w(t,x)|_{_{_{t=0}}}&=w_0(x), \\
\partial_t w(t,x)|_{_{_{t=0}}}&=w_1(x),
\end{align*}
where $^{C}\partial_{t}^{\beta}w(t,x)=\prescript{RL}{0}I^{2-\beta}\partial_t^{(2)} w(t,x).$
\item $\mathscr{L}$-Schr\"odinger type equations:
\begin{align*}
i\,^{C}\partial_{t}^{\beta}w(s,x)-\mathscr{L}w(t,x)&=0,\quad t>0,\quad x\in G,\quad 0<\beta<1, \\
w(t,x)|_{_{_{t=0}}}&=w_0(x). 
\end{align*}
\end{itemize}
Note that for $\beta=2$ we get the classical partial derivative in time, i.e. $^{C}\partial_{t}^{\beta}w(t,x)=\prescript{RL}{0}I^{0}\partial_t^{(2)} w(t,x)=\partial_t^{(2)} w(t,x)$ since $\prescript{RL}{0}I^{0}$ acts like the identity operator. Also, for $\beta=1$, we obtain the classical partial derivative in time $^{C}\partial_{t}^{\beta}w(t,x)=\partial_t w(t,x).$
 
For all the above mentioned equations, the solutions (respectively their propagators) are closely related with the Mittag-Leffler function
\begin{equation*}
E_{\alpha,\delta}(z)=\sum_{k=0}^{+\infty} \frac{z^k}{\Gamma(\alpha k+\delta)},\quad z,\delta\in\mathbb{C},\quad \text{\rm Re}(\alpha)>0,
\end{equation*}
which is absolutely and locally uniformly convergent for the given parameters (\cite{mittag}). Now, by using the Borel functional calculus, the propagators for heat and wave types above can be represented by 
\[
E_{\beta}(-t^{\beta}\mathscr{L})\,\,(0<\beta<1)\quad\text{and}\quad E_{\beta}(-t^{\beta}\mathscr{L}),\,\, \prescript{RL}{0}I^{1}_t E_{\beta}(-t^{\beta}\mathscr{L})\,\, (1<\beta<2). 
\]
Moreover, the solution for the Schr\"odinger type equation is given by $E_\beta(it^{\beta}\mathscr{L}).$ It is known that $E_\beta(-t^{\beta}x)$ $(t,x>0)$ is completely monotonic for all $0<\beta\leqslant1$ \cite{Pollard}. While, for the range $1<\beta<2$, it can be shown to have some oscillatory behaviour. Also, $\mathscr{L}$ is a positive sectorial operator \cite[Chapter 2]{functionalcalculus}. Thus, an equivalent form to represent the propagator $E_\beta(-t^{\beta}\mathscr{L})$ is given by \cite[Theorem 2.41]{thesis}
	\[
	E_\beta(-t^{\beta}\mathscr{L})=\frac{1}{2\pi i}\int_{H}e^{\gamma t}\gamma^{\beta-1}(\gamma^{\beta}+\mathscr{L})^{-1}d\gamma,\quad t\geqslant0,\quad 0<\beta<1,
	\]
	where $H\subset\rho(-\mathscr{L})$ and $H$ is a suitable Hankel's path.

Below we provide some graphs (Figures from \ref{grafica1} to \ref{grafica4}) of the oscillating behaviour, by varying the values of $\beta$, and also show that these type of functions are always bounded uniformly by $C/(1+x)$ for some positive constant $C$ (see formula \eqref{uniform-estimate}): 

\begin{center}
    \begin{figure}[ht]
        \centering
        \includegraphics[scale=0.69]{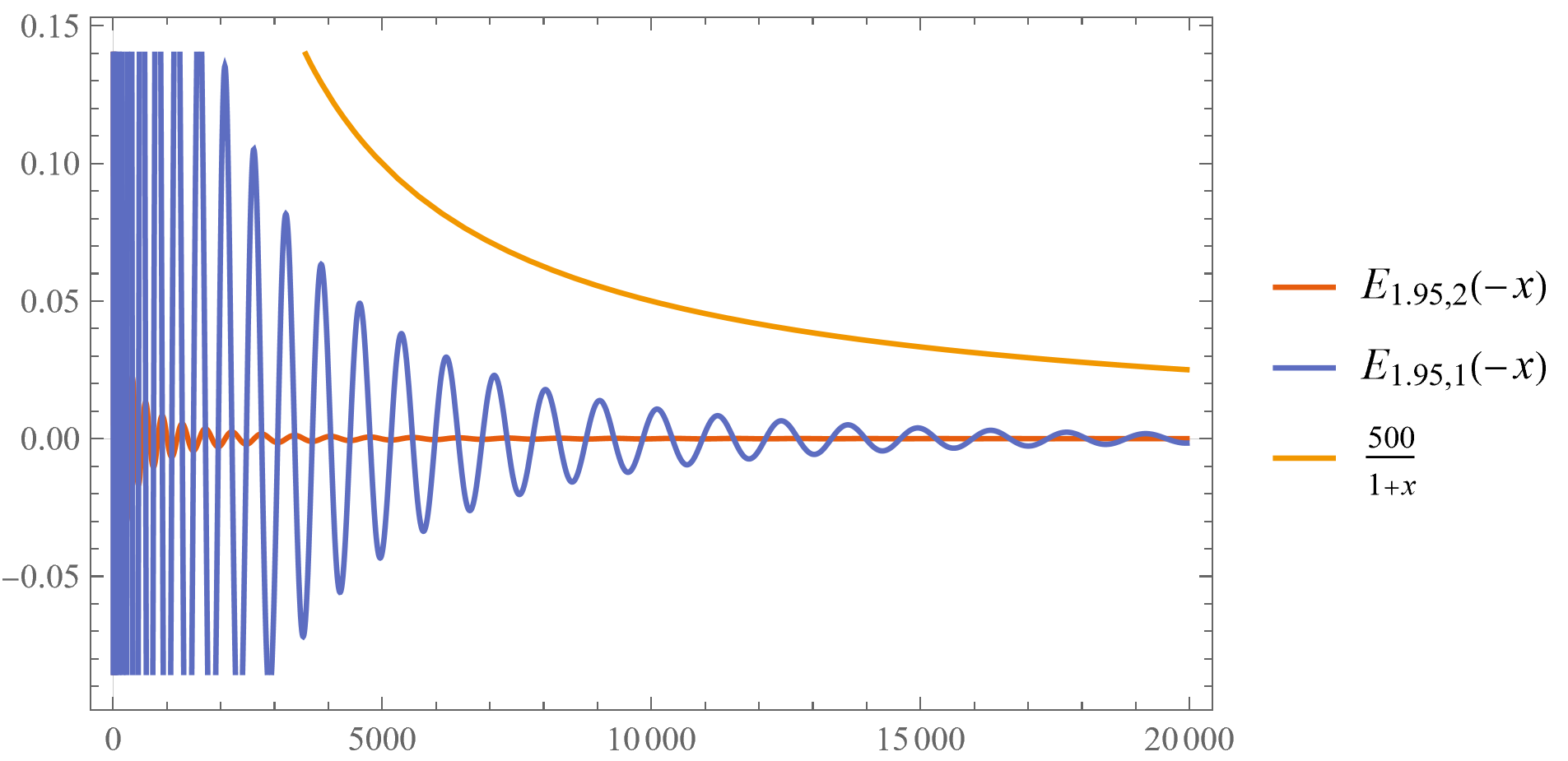}
        \caption{Wave type propagators functions for $\beta=1.95$ bounded uniformly.}
        \label{grafica1}
    \end{figure}
\end{center}
\begin{center}
    \begin{figure}[ht]
        \centering
        \includegraphics[scale=0.69]{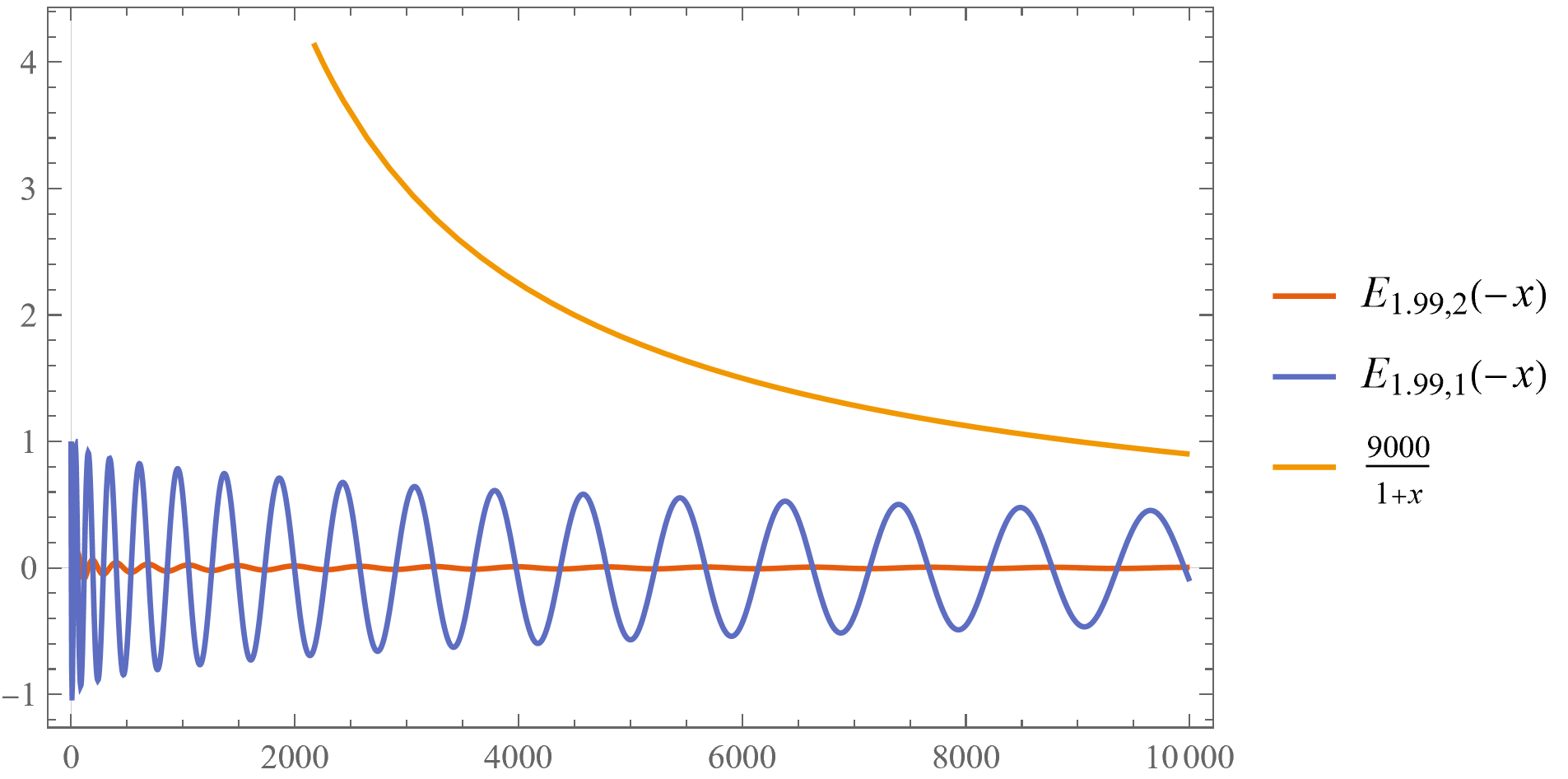}
        \caption{Wave type propagators functions for $\beta=1.99$ bounded uniformly (oscillation much greater than when $\beta=1.95$).}
        \label{grafica2}
    \end{figure}
\end{center}
\begin{center}
    \begin{figure}[ht]
        \centering
        \includegraphics[scale=0.69]{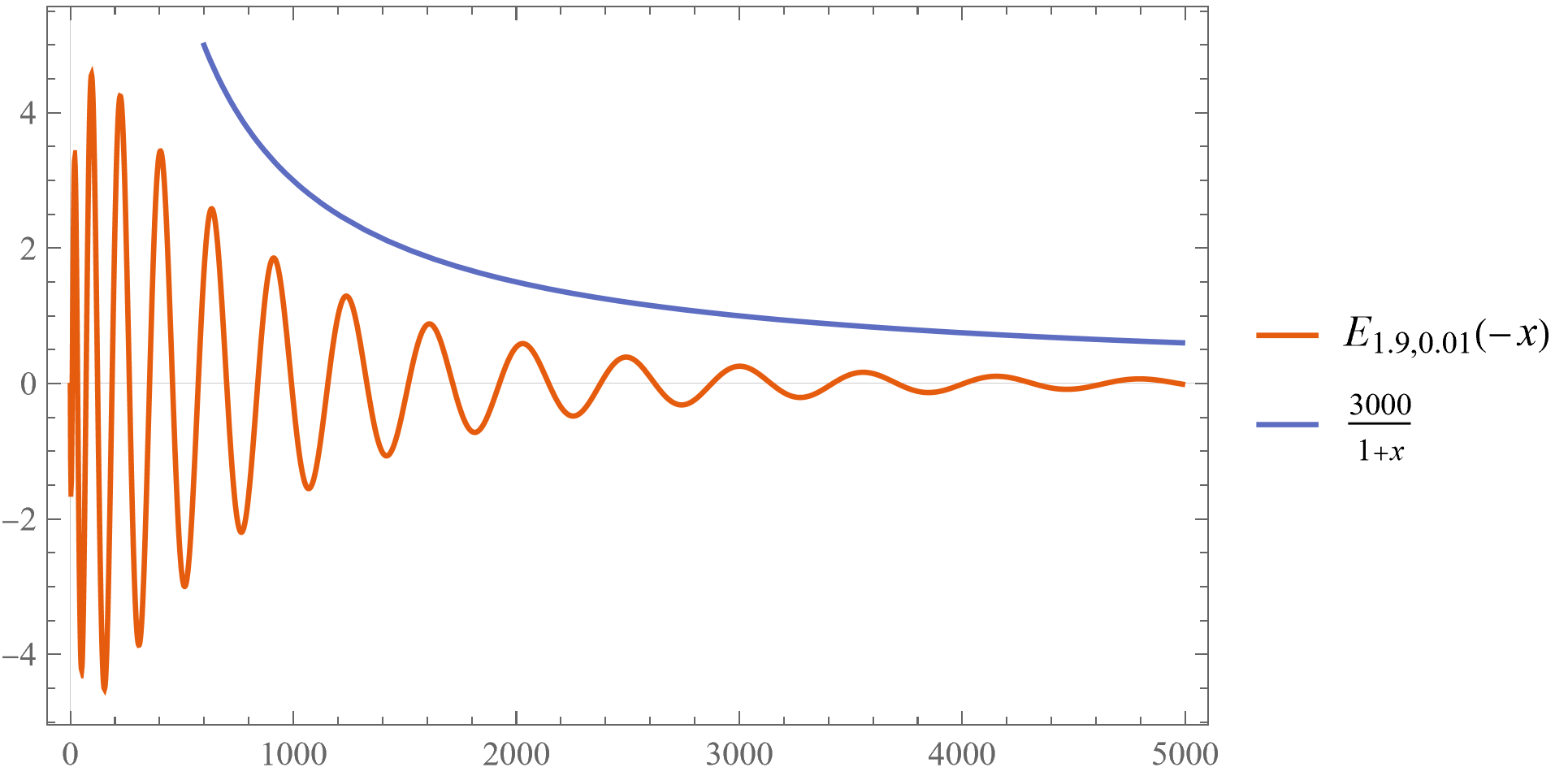}
        \caption{Mittag-Leffler functions with one small parameter.}
        \label{grafica3}
    \end{figure}
\end{center}
\begin{center}
    \begin{figure}[ht]
        \centering
        \includegraphics[scale=0.69]{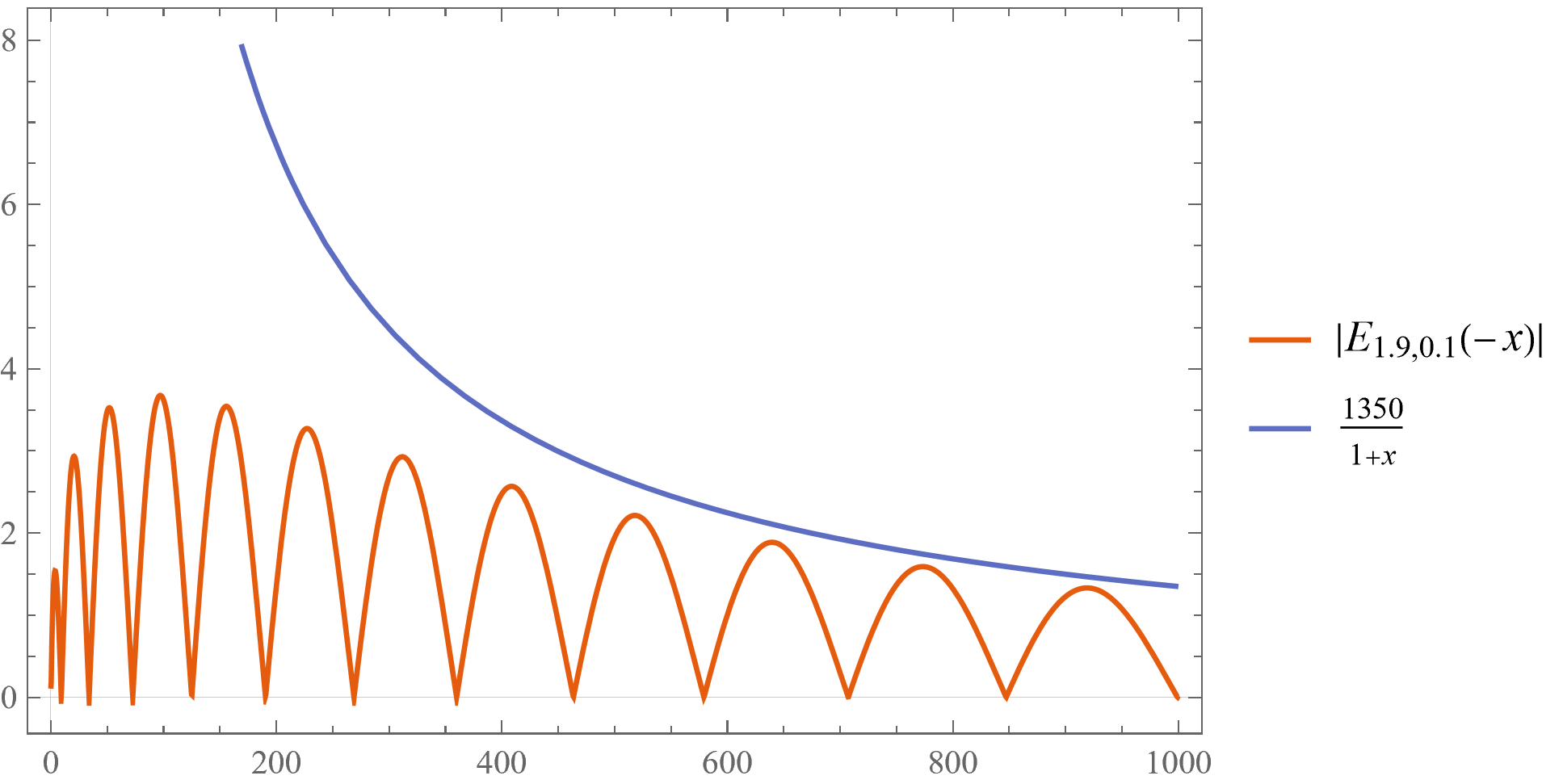}
        \caption{Absolute value of Mittag-Leffler functions remains bounded.}
        \label{grafica4}
    \end{figure}
\end{center}
So, in particular, if condition \eqref{intro-trace} holds, by Theorem \ref{intro-lp-lq}, we get the following time decay rate for the solutions of $\mathscr{L}$-heat type equation and $\mathscr{L}$-Schr\"odinger type equation:  
\[
\|w(t,\cdot)\|_{L^q(G)}\leqslant C_{\beta  ,\lambda,p,q}t^{-\beta\lambda\left(\frac{1}{p}-\frac{1}{q}\right)}\|w_0\|_{L^p(G)},   
\]
and also for the $\mathscr{L}$-wave type equation:
\[
\|w(t,\cdot)\|_{L^q(G)}\leqslant  C_{\beta,\lambda,p,q}t^{-\beta\lambda\left(\frac{1}{p}-\frac{1}{q}\right)} \big(\|w_0\|_{L^p(G)}+t\|w_1\|_{L^p(G)}\big),   
\]
whenever $\frac{1}{\lambda}\geqslant\frac{1}{p}-\frac{1}{q}.$ Note that the order $(\beta)$ of the fractional derivative is appearing in the order of the decay in each case.

Now, in this paper, we have realized that some of the above equations belong to a special class of equations with general kernels called evolutionary integral equations of scalar type. Next, we recall the abstract setting of these type of equations. Thus, we study the following equation
\begin{equation}\label{intro-volterra-e}
w(t)=h(t)+\int_0^t k(t-s)\mathscr{L}w(s){\rm d}s,\quad t\in [0,T],    
\end{equation}
where $\mathscr{L}$ is a closed linear unbounded operator in $X$ (a complex Banach space) with dense domain $\mathcal{D}(\mathscr{L})$, $h\in C([0,T];X)$ and $k\in L_{loc}^1(\mathbb{R}_+)$ is a scalar kernel different from $0$. 

Therefore, in Section \ref{regularity}, we discuss the above Volterra equation. Our space (base) will be to consider the Banach space $X=L^p(G)$ $(1\leqslant p\leqslant+\infty)$ where $G$ is a locally compact group. Moreover, we will consider operators that are positive and left invariant on $G.$

We first state some results for the above general equations. Later, we just focus in a class of kernels called \textit{completely positive} $(\mathcal{PC})$, i.e. $k\in\mathcal{PC}$ if $k\in L_{loc}^1(\mathbb{R}_+)$ is nonnegative and nonincreasing, and there exists a nonnegative kernel $\mathscr{K}\in L_{loc}^1(\mathbb{R}_+)$ such that 
\begin{equation}\label{laplace}
\big(\mathscr{K}\ast k\big)(t):=\int_0^t \mathscr{K}(t-s)k(s){\rm d}s=1,\quad \text{on}\quad (0,+\infty).
\end{equation}
The latter condition is well-known as the Sonine condition \cite{sonine}. These classes were first introduced by Cl\'ement and Nohel in \cite{[50],Clement}.

One of our main results on these type of equations reads as:

\begin{thm}\label{intro-main-integral}
Let $G$ be a locally compact separable unimodular group and let $1<p\leqslant 2\leqslant q<+\infty$. Let $\mathscr{L}$ be any left invariant operator on $G$. Assume that $k\in\mathcal{PC}$ is positive and  
\begin{equation*}
\sup_{t>0}\sup_{v>0}\big[\tau\big(E_{(0,v)}(|\mathscr{L}|)\big)\big]^{\frac{1}{p}-\frac{1}{q}}\frac{1}{1+v\int_0^t k(\tau){\rm d}\tau}<+\infty. 
\end{equation*}
If $w_0\in L^p(G)$ then the solution of the integral equation \eqref{intro-volterra-e} is in $L^q(G)$.

Also, in the particular case that condition \eqref{intro-trace} holds, we obtain the following time decay rate for the solution of equation \eqref{intro-volterra-e} 
\[
\|w(t,\cdot)\|_{L^q(G)}\leqslant C_{\lambda,p,q}\bigg(\int_0^t k(\tau){\rm d}\tau\bigg)^{-\lambda\left(\frac{1}{p}-\frac{1}{q}\right)}\|w_0\|_{L^p(G)},\quad \frac{1}{\lambda}\geqslant\frac{1}{p}-\frac{1}{q}.    
\]
\end{thm}
We also complement our studies by studying the following $\mathscr{L}$-evolutionary differential equation:
\begin{equation}\label{intro-differential-lcg}
\begin{split}
\partial_t\big(k\ast[w(s,x)-w_0(x)]\big)(t)+\mathscr{L}w(t,x)&=0, \quad t>0,\quad x\in G,\quad k\in \mathcal{PC}, \\
w(t,\cdot)|_{_{_{t=0}}}&=w_0\in L^p(G),
\end{split}
\end{equation}
where $\mathscr{L}$ is a closed linear unbounded operator in $L^p(G)$ with dense domain $\mathcal{D}(\mathscr{L})$. In this case, the idea is to transform the differential equation in an integral one that allows us to use the established result in Theorem \ref{intro-main-integral}. Here we need that the propagator $S(t)$ for the integral equation to be (at least) differentiable, therefore we have to assume additionally that $\mathscr{K}$ is positive and $\mathscr{K}\in BV_{loc}(\mathbb{R}_+)$ (bounded variation functions) \cite[Proposition 1.2]{pruss}. Note that in \cite{pruss}, the author normalized the functions $k$ in latter class by considering $k(0)=0$ and $k(\cdot)$ is left-continuous on $\mathbb{R}_+.$ Nevertheless, in our case, it is not necessary to assume it. Thus, we have that: 
\begin{thm}\label{intro-differential-equation}
Let $(k,\mathscr{K})\in \mathcal{PC}$ be such that $\mathscr{K}$ is positive and $\mathscr{K}\in BV_{loc}(\mathbb{R}_+).$ Let $G$ be a locally compact separable unimodular group and let $1<p\leqslant 2\leqslant q<+\infty$. Let $\mathscr{L}$ be any positive left invariant operator on $G$. Suppose that 
\begin{equation*}
\sup_{t>0}\sup_{v>0}\big[\tau\big(E_{(0,v)}(|\mathscr{L}|)\big)\big]^{\frac{1}{p}-\frac{1}{q}}\frac{1}{1+v\int_0^t \mathscr{K}(\tau){\rm d}\tau}<+\infty. 
\end{equation*}
If $w_0\in L^p(G)$ then the solution of the partial integro-differential equation \eqref{intro-differential-lcg} is in $L^q(G).$

Also, in the particular case that condition \eqref{intro-trace} holds, we get the following time decay rate for the solution of equation \eqref{intro-differential-lcg} 
\[
\|w(t,\cdot)\|_{L^q(G)}\leqslant C_{\lambda,p,q}\bigg(\int_0^t \mathscr{K}(\tau){\rm d}\tau\bigg)^{-\lambda\left(\frac{1}{p}-\frac{1}{q}\right)}\|w_0\|_{L^p(G)},\quad \frac{1}{\lambda}\geqslant\frac{1}{p}-\frac{1}{q}.    
\]
\end{thm}

Notice now that many of our asymptotic results are predetermined by condition \eqref{intro-trace}. Hence, a natural question is to know which type of operators can be considered here? In Subsection \ref{type of operators}, we discuss it in detail. In fact, we show that operators like Laplacians, sub-Laplacians, Rockland operators, etc, arisen and used in different groups satisfy such a condition. While, in Subsection \ref{Examples}, we give several examples of different equations like multi-term $\mathscr{L}$-heat type equations, $\mathscr{L}$-Cauchy problem with time variable coefficient, homogeneous Rayleigh–Stokes type problem, etc. 

In Section \ref{strichartz}, we focus on an application of the $L^p-L^q$ results that we have obtained in the other sections. This is, we prove local well-posedness of some nonlinear partial integro-differential equations. The strategy consists on studying first the non-homogeneouos equation for a function $f$ that not necessarily is nonlinear, so  obtaining boundedness on the mixed-type normed spaces $L_t^r L_x^q([0,T] \times G)$, and then utilize these new estimates to perform a Banach fixed-point type argument in an appropriate Banach space $S([0,T]\times G)$. Let us mention that this kind of analysis started in the seminal work of Strichartz \cite{Stri} for the so-called dispersive equations which have the classical Schrödinger-wave equations as key representatives. Moreover, a major contribution on that direction was made by Keel and Tao \cite{KeelTao} where in a very abstract set-up they studied the problem of endpoints. At this moment we can not cover those equations, but we can work on the classical heat equation, and its time-fractional counterpart, as well as for the time-fractional wave equation. On $\mathbb{R}^n$, one can see results on the heat equation in e.g. \cite{bao,EstCauchTimeSpac}. Our results are as follows: 
\begin{itemize}
    \item Nonlinear $\mathscr{L}$-heat equations:
    
\begin{align}\label{NLHeatIntr}
\begin{split}
\partial_{t}w(t,x)+\mathscr{L}w(t,x)&=F(t,w), \quad t>0,\quad x\in G, \\
w(t,x)|_{_{_{t=0}}}&=w_0(x).
\end{split} 
\end{align} 
\begin{thm}
    Let $1<p_0\leqslant 2$. Suppose that the operator $\mathscr{L}$ satisfies the condition \eqref{intro-trace}. Suppose that there exist $0<T<+\infty$ and $1\leqslant\rho\leqslant+\infty$ such that the nonlinearity $F$ satisfies the following estimate:
    \begin{equation*}
        \|F(u) - F(v)\|_{L_t^{\rho} L_x^{p_0}([0,T] \times G)}\leqslant \frac{1}{2C_{\lambda,p_0,\rho,T}} \|u -v\|_{S_{H}^{p_0}([0,T]\times G)},
    \end{equation*}
    for all $u,v\in B_{\varepsilon}:= \{u\in S_{H}^{p_0}([0,T]\times G): \|u\|\leqslant \varepsilon\}$, for some $\varepsilon>0$. 
    Then the problem \eqref{NLHeatIntr} is locally well-posed in $L_x^{p_0}(G)$. 
\end{thm}
\item Nonlinear $\mathscr{L}$-heat type equations:
\begin{align}\label{NLHeatTIntr}
\begin{split}
\,^{C}\partial_{t}^{\beta}w(t,x)+\mathscr{L}w(t,x)&=F(t,w), \quad t>0,\quad x\in G,\quad 0<\beta<1, \\
w(t,x)|_{_{_{t=0}}}&=w_0(x).
\end{split}
\end{align}
\begin{thm}
    Let $1<p_0\leqslant 2$. Suppose that the operator $\mathscr{L}$ satisfies the condition \eqref{intro-trace}. Suppose that there exist $0<T<+\infty$ and $\frac{1}{\beta}\leqslant\rho\leqslant+\infty$ such that the nonlinearity $F$ satisfies the following estimate
    \begin{equation*}
        \|F(u) - F(v)\|_{L_t^{\rho} L_x^{p_0}([0,T] \times G)}\leqslant\frac{1}{2C_{\beta,\lambda,p_0,\rho,T}} \|u -v\|_{S_{HT}^{p_0}([0,T]\times G)},
    \end{equation*}
    for all $u,v\in B_{\varepsilon}:= \{u\in S_{HT}^{p_0}([0,T]\times G): \|u\|\leqslant \varepsilon\}$, for some $\varepsilon>0$. 
    Then the problem \eqref{NLHeatTIntr} is locally well-posed in $L_x^{p_0}(G)$. 
\end{thm}
\item Nonlinear $\mathscr{L}$-wave type equations:
\begin{align}\label{NLwaveIntr}
\begin{split}
\,^{C}\partial_{t}^{\beta}w(t,x)+\mathscr{L}w(t,x)&=F(t,w), \quad t>0,\quad x\in G,\quad 1<\beta<2, \\
w(t,x)|_{_{_{t=0}}}&=w_0(x), \\
\partial_t w(t,x)|_{_{_{t=0}}}&=w_1(x).
\end{split}
\end{align}
\begin{thm}
    Let $1<p_0\leqslant2$. Suppose that the operator $\mathscr{L}$ satisfies the condition \eqref{intro-trace}. Suppose that there exist $0<T<+\infty$ and $1\leqslant\rho\leqslant+\infty$ such that the nonlinearity $F$ satisfies the following estimate
    \begin{equation*}
        \|F(u) - F(v)\|_{L_t^{\rho} L_x^{p_0}([0,T] \times G)}\leqslant\frac{1}{2\Tilde{C}_{\beta,\lambda,p_0,\rho,T}} \|u -v\|_{S_{W}^{p_0}([0,T]\times G)},
    \end{equation*}
    for all $u,v\in B_{\varepsilon}:= \{u\in S_{W}^{p_0}([0,T]\times G): \|u\|\leqslant \varepsilon\}$, for some $\varepsilon>0$. 
    Then the problem \eqref{NLwaveIntr} is locally well-posed in $L_x^{p_0}(G)\times L_x^{p_0}(G)$.
\end{thm}
\end{itemize}

\section{Preliminary results}\label{preli}

In this section we collect all the necessary concepts and results about evolutionary integral-differential equations (in abstract settings), and also about von Neumann algebras, that will be used everywhere in this paper. 

\subsection{Notions on evolutionary integral equations}

\subsubsection{$\mathscr{L}$-evolutionary integral equation of scalar type}

We study the following Volterra equation
\begin{equation}\label{volterra-e}
w(t)=h(t)+\int_0^t k(t-s)\mathscr{L}w(s)\,{\rm d}s,\quad t\in [0,T],    
\end{equation}
where $\mathscr{L}$ is a closed linear unbounded operator in $X$ (a complex Banach space) with dense domain $\mathcal{D}(\mathscr{L})$, $h\in C([0,T];X),$ and $k\in L_{loc}^1(\mathbb{R}_+)$ is a scalar kernel different from $0$. 

Let us now recall some notions for the solution of \eqref{volterra-e} (\cite{pruss}). Below we denote by $X_{\mathscr{L}}$ the domain of $\mathscr{L}$ endowed with the graph norm $|\cdot|_{\mathscr{L}}$ of $\mathscr{L}$, i.e. $|v|_{\mathscr{L}}=|v|+|\mathscr{L}v|.$ 

\begin{defn}
Let $w$ be a function in the space $C([0,T];X)$. It is said that $w$ is:
\begin{enumerate}
\item {\it strong solution} of \eqref{volterra-e} on $[0,T]$ if $w\in C([0,T];X_{\mathscr{L}})$ and equation \eqref{volterra-e} holds for all $t\in [0,T]$;
\item {\it mild solution} of \eqref{volterra-e} on $[0,T]$ if $k\ast w\in C([0,T];X_{\mathscr{L}})$ and
\[
w(t)=h(t)+\mathscr{L}\big(k\ast w\big)(t),\quad \text{for all}\quad t\in [0,T];
\]
\item {\it weak solution} of \eqref{volterra-e} on $[0,T]$ if 
\[
\langle w(t),y\rangle =\langle h(t),y\rangle +\langle (k\ast w)(t),\mathscr{L}^*y\rangle,
\]
for all $t\in [0,T]$ and each $y\in \mathcal{D}(\mathscr{L}^*).$
\end{enumerate}
\end{defn}
Notice that each strong solution of \eqref{volterra-e} is a mild solution, and every mild solution is a weak one. Now we discuss the well-posedness of \eqref{volterra-e}, which is an extension of the classical notion of well-posed Cauchy problems.

\begin{defn}
Equation \eqref{volterra-e} is said to be well-posed if the following two conditions are satisfied: 
\begin{enumerate}
    \item For each $v\in\mathcal{D}(\mathscr{L})$ there exists a unique strong solution $w_v$ on $\mathbb{R}_+$ of 
    \[
  w(t)=v+(k\ast \mathscr{L}w)(t),\quad t\geqslant0.  
    \]
    \item For all sequences  $\{v_{m}\}_{m\geqslant1}\subset\mathcal{D}(\mathscr{L})$, such that $v_{m}\to0$ as $m\to+\infty$ in $X$, imply $w_{v_m}(t)\to0$ as $m\to+\infty$ in $X$, uniformly on compact intervals.  
\end{enumerate} 
\end{defn}
Without loss of generality, we frequently denote $w_v(t)$ by $w(t).$ Assume that \eqref{volterra-e} is well-posed, we can then define the solution operator $S(t)$ for \eqref{volterra-e} as follows:
\[
w(t)=S(t)v,\quad v\in\mathcal{D}(\mathscr{L}),\quad t\geqslant 0.
\]
Therefore we can also work with the following characterization as well.
\begin{defn}
A family $\{S(t)\}_{t\geqslant0}\subset \mathcal{B}(X)$ (bounded linear operators in $X$) is called a {\it solution operator (or resolvent)} of equation \eqref{volterra-e} if the following conditions are satisfied:
\begin{enumerate}
    \item $S(t)$ is strongly continuous for $t\geqslant0$ and $S(0)=I;$
    \item $S(t)\mathcal{D}(\mathscr{L})\subset\mathcal{D}(\mathscr{L})$ and $\mathscr{L}S(t)v=S(t)\mathscr{L}v$ for any $v\in\mathcal{D}(\mathscr{L})$, $t\geqslant0;$
    \item The resolvent equation holds
    \[
    S(t)v=v+\int_0^t k(t-s)\mathscr{L}S(s)v\,{\rm d}s,\quad\text{for any}\quad v\in\mathcal{D}(\mathscr{L}),\,\,t\geqslant0.
    \]
\end{enumerate}
\end{defn}
It is known that a well-posed problem \eqref{volterra-e} admits a solution operator $S(t)$, and vice versa \cite[Proposition 1.1]{pruss}. So, we now can use the fact that equation \eqref{volterra-e} is well-posed if and only if it has a resolvent $S(t)$ satisfying 
\begin{equation}\label{resolvent-e}
S(t)v=v+\mathscr{L}\int_0^t k(t-s)S(h)v\,{\rm d}h, \quad\text{for all}\quad v\in X,\quad t\geqslant 0. 
\end{equation}

Through the paper we frequently denote $s(t;\lambda)$ to be the one dimensional resolvent by considering $\lambda>0$ to be the operator $\mathscr{L}$, i.e. $\lambda\equiv \mathscr{L}$ and it satisfies 
\begin{equation}\label{resolvent-e-1}
s(t;\lambda)=1+\lambda\int_0^t k(t-s)s(h;\lambda){\rm d}h,\quad t\geqslant 0.    
\end{equation}

\subsection{Von Neumann algebras} The starting point of these algebras can be found in the well-known papers of von Neumann in \cite{[46],[47]}, where they built the mathematical foundations for studying quantum mechanics. Let $\mathfrak{L}(\mathcal{H})$ be the set of linear operators defined on a Hilbert space $\mathcal{H}$. Notice that the concept of $\tau$-measurability on a von Neumann algebra $M$ and the technique of spectral projections allows us to approximate unbounded operators by bounded ones. 

For our study we just need to set $M$ to be the right group von Neumann algebra $VN_R(G)$ with $G$ being a locally compact separable unimodular group. The latter assumption gives the possibility to use the notion of noncommutative Lorentz spaces on $M$ as given in \cite{[51]}. For more specific details on von Neumann algebras we refer \cite{von1,von,von2}.  

It is important to recall that the group von Neumann algebra $VN_R(G)$ is generated by all the right actions of $G$ on $L^2(G)$ ($\pi_R(g)f(x)=f(xg)$ with $g\in G$), which means that $VN_R(G)=\{\pi_R(G)\}^{!!}_{g\in G}$, where $!!$ is the bicommutant of the self-adjoint subalgebras $\{\pi_R(g)\}_{g\in G}\subset \mathfrak{L}(L^2(G))$. The latter result is a consequence of the fact \cite{von}: $VN_R(G)^{!}=VN_L(G)$ and $VN_L(G)^{!}=VN_R(G)$, where the symbol $!$ represents the commutant. 

Before recalling the noncommutative Lorentz spaces we need the following definitions.  

\begin{defn}
Let $M\subset\mathfrak{L}(\mathcal{H})$ be a semifinite von Neumann algebra acting over the Hilbert space $\mathcal{H}$ with a trace $\tau$. A linear operator $L$ (maybe unbounded) is said to be affiliated with $M$, if it commutes with the elements of the commutant $M^!$ of $M$, which means $LV=VL$ for any $V\in M^!.$
\end{defn}

We highlight that in this paper we will consider operators affiliated with $VN_R(G)$ (see \cite[Def. 2.1]{RR2020}), which are precisely those ones who are left invariant on $G$ \cite[Remark 2.17]{RR2020}.

\begin{rem}
If $L$ is a bounded operator affiliated  with $M$, then $L\in M$ by the double commutant theorem. 
\end{rem}

So, the next definitions are given for this type of operators.   

\begin{defn}
Let $M$ be a von Neumann algebra. A trace of the positive part $M_+=\{L\in M: L^*=L>0\}$ of $M$ is a functional defined on $M_+$, taking non-negative, possibly infinite, real values, with the following properties:
\begin{enumerate}
    \item If $L\in M_+$ and $T\in M_+$ then $\tau(L+T)=\tau(L)+\tau(T)$;
    \item If $L\in M_+$ and $\gamma\in\mathbb{R}^{+}$ then $\tau(\gamma L)=\gamma\tau(L)$ (with $0\cdot+\infty=0$);
    \item If $L\in M_+$ and $U$ is an unitary operator on $M$ then $\tau(ULU^{-1})=\tau(L).$
\end{enumerate}
\end{defn}

A trace $\tau$ is faithful (or exact) if $\tau(L)=0$ $(L\in M_+)$ implies that $L=0$. We have that $\tau$ is finite if $\tau(L)<+\infty$ for all $L\in M_+$. Also, $\tau$ is semifinite if, for each $L\in M_+$, $\tau(L)$ is the supremum of the numbers $\tau(T)$ over those $T\in M_+$ such that $L\leqslant T$ and $\tau(T)<+\infty.$  

\begin{defn}
A closeable operator $L$ (maybe unbounded) is called  $\tau$-measureable if for each $\epsilon>0$ there exists a projection $P$ in $M$ such that $P(\mathcal{H})\subset D(L)$ and $\tau(I-P)\leqslant \epsilon$, where $D(L)$ is the domain of $L$ in $\mathcal{H}$ and $M\subset\mathfrak{L}(\mathcal{H})$. We denote by $S(M)$ the set of all $\tau$-measurable operators.   
\end{defn}

\begin{defn}
Let us take an operator $L\in S(M)$ and let $L = U|L|$ be its polar decomposition. We define the distribution function by $d_\gamma(L):=\tau\big(E_{(\gamma,+\infty)}(|L|)\big)$ for $\gamma\geqslant0$, where $E_{(\gamma,+\infty)}(|L|)$ is the spectral projection of $L$ over the interval $(\gamma,+\infty).$ Also, for any $t>0$, we define the generalized $t$-th singular numbers as 
\[
\mu_t(L):=\inf\{\gamma\geqslant0:\,d_\gamma(L)\leqslant t\}.
\]
\end{defn}
For more details and properties of the distribution function and generalized singular numbers, we recommend to see \cite{pacific}. 

Below we recall the noncommutative Lorentz spaces associated with a semifinite von Neumann algebra $M$, which are a noncommutative extension of the classical Lorentz spaces \cite{[51]}.

\begin{defn}
We denote by $L^{p,q}(M)$ $(1\leqslant p<+\infty,\,1\leqslant q<+\infty)$ the set of all operators $L\in S(M)$ such that
\begin{align*}
\|L\|_{L^{p,q}(M)}=\left(\int_0^{+\infty}\big(t^{1/p}\mu_t(L)\big)^q \frac{{\rm d}t}{t}\right)^{1/q}<+\infty.    
\end{align*}
Notice that the $L^{p}$-spaces on $M$ can be defined by
\[
\|L\|_{L^{p}(M)}:=\|L\|_{L^{p,p}(M)}=\left(\int_0^{+\infty}\mu_t(L)^p{\rm d}t\right)^{1/p}.
\]
In the case $q=+\infty$, $L^{p,\infty}(M)$ is the set of all operators $L\in S(M)$ such that 
\[
\|L\|_{L^{p,\infty}(M)}=\sup_{t>0}t^{1/p}\mu_t(L)<+\infty.
\]
\end{defn}

\section{Regularity of integral (differential) equations of scalar type}\label{regularity} 

In this section, we discuss existence, uniqueness, norm estimates and asymptotic time decay for an evolutionary integral (differential) equation of scalar type on a locally compact group. These norm estimates can be reduced to the calculation of its propagator in the noncommutative Lorentz space norm. Also, we show that the latter norm mainly involves to estimate the trace of the spectral projections of the considered operator. 

Note that in the previous paper \cite[Theorem 1]{SRR}, the authors derived the following useful result:

\begin{thm}\label{additional}
Let $\mathscr{L}$ be a closed (maybe unbounded) operator affiliated with a semifinite von Neuman algebra $M$. Let $\phi$ be a Borel measurable function on $[0,+\infty)$. Suppose also that $\psi$ is a monotonically decreasing continuous function on $[0,+\infty)$ such that $\psi(0)>0$, $\displaystyle\lim_{s\to+\infty}\psi(s)=0$ and $|\phi(s)|\leqslant \psi(s)$ for all $s\in[0,+\infty).$ Then for every $1\leqslant r<+\infty$ we have the inequality 
\[
\|\phi(|\mathscr{L}|)\|_{L^{r, \infty}(M)}\leqslant \sup_{s>0}\psi(s)\big[\tau(E_{(0,s)}(|\mathscr{L}|))\big]^{\frac{1}{r}}.
\]
\end{thm}

At the first step we just showed some applications of above result in heat-wave-Schr\"odinger type equations (time fractional versions of the classical ones). Nevertheless, we have now figured out that the latter result is more efficient and therefore we are able to apply it to a general evolution integral (differential) equation of scalar type. Hence, we provide the results and examples in the next lines.

As a consequence of Theorem \ref{additional} and \cite[Corollary 6.3]{RR2020} we also have in a  straightforward way the following statement:  
\begin{cor}\label{coro-lplq}
Let $G$ be a locally compact separable unimodular group and let $1<p\leqslant 2\leqslant q<+\infty$. Let $\mathscr{L}$ be any left invariant operator on $G$ (maybe unbounded). Let $\phi$ be a Borel measurable function on $[0,+\infty)$. Suppose also that $\psi$ is a monotonically decreasing continuous function on $[0,+\infty)$ such that $\psi(0)>0$, $\displaystyle\lim_{s\to+\infty}\psi(s)=0$ and $|\phi(s)|\leqslant \psi(s)$ for all $s\in[0,+\infty).$ It follows that 
\[
\|\phi(|\mathscr{L}|)\|_{L^{p}(G)\rightarrow L^q(G)}\lesssim\sup_{s>0}\psi(s)\big[\tau(E_{(0,s)}(|\mathscr{L}|))\big]^{\frac{1}{p}-\frac{1}{q}}.
\]    
\end{cor}

Usually, for the calculation (also existence) of the supremum in Theorem \ref{additional} (or Corollary \ref{coro-lplq}), it is assumed the following condition holds
\begin{equation}\label{trace-condition}
\tau\big(E_{(0,s)}(|\mathscr{L}|)\big)\lesssim s^{\lambda},\quad s\to+\infty,\quad\text{for some}\quad \lambda>0,
\end{equation}
which will be discussed with more details in Subsection \ref{type of operators}.

\subsection{$\mathscr{L}$-evolutionary integral equation}\label{integral-equation} We consider the following integral  equation of scalar type on $L^p(G)$ $(1\leqslant p\leqslant+\infty$, $G$ is a locally compact separable unimodular group): 
\begin{equation}\label{volterra-e-locally}
w(t,x)=h(t,x)+\int_0^t k(t-s)\mathscr{L}w(s,x){\rm d}s,\quad t\in [0,T],\quad x\in G,    
\end{equation}
where $h\in C([0,T];L^p(G))$, $\mathscr{L}$ is a closed linear unbounded operator in $L^p(G)$ with dense domain $\mathcal{D}(\mathscr{L})$ and $k\in L_{loc}^1(\mathbb{R}_+)$ is a scalar kernel different from $0$.

In this section we normally assume the existence of a resolvent $S(t)$ in $\mathcal{B}(L^p(G))$ (bounded linear operators in $L^p(G)$) for equation \eqref{volterra-e-locally}. Nevertheless, in some examples, it will be discussed the existence of such operator. More details about resolvents can be found in e.g. \cite[Chapter 1, Section 1.2]{pruss}. Note that the existence of a resolvent guarantees the well-posed of \eqref{volterra-e-locally} and vice versa \cite[Proposition 1.1]{pruss}. 

\medskip Below we establish the $L^p(G)-L^q(G)$ $(1<p\leqslant 2\leqslant q<+\infty)$ boundedness for the solution of equation \eqref{volterra-e-locally}.  

\begin{thm}\label{integral-thm-0}
Let $G$ be a locally compact separable unimodular group and $1<p\leqslant 2\leqslant q<+\infty$. Let $\mathscr{L}$ be any left invariant operator on $G$. Suppose that $s(t;\lambda)\leqslant C\psi(t;\lambda)$ for a monotonically decreasing continuous function with respect to the variable $\lambda$ on $(0,+\infty)$ such that $0<\psi(t;0)<+\infty$ and $\displaystyle\lim_{\lambda\to+\infty}\psi(t;\lambda)=0$ uniformly on $t.$ Assume also that 
\begin{equation}\label{need-general-0}
\sup_{t>0}\sup_{v>0}\big[\tau\big(E_{(0,v)}(|\mathscr{L}|)\big)\big]^{\frac{1}{p}-\frac{1}{q}}\psi(t;v)<+\infty. 
\end{equation}
If $w_0\in L^p(G)$ then the solution of the integral equation \eqref{volterra-e-locally} is in $L^q(G)$. 
\end{thm}
\begin{proof}
    In fact, by \cite[Theorem 5.1]{RR2020}, we have that
\begin{equation}\label{volterra-sol}
\|w(t,\cdot)\|_{L^q(G)}=\|S(t)w_0\|_{L^q(G)}\lesssim \|S(t)\|_{L^{r,\infty}(VN_R(G))}\|w_0\|_{L^p(G)},
\end{equation}
where $\frac{1}{r}=\frac{1}{p}
-\frac{1}{q}$. Since $s(t;\lambda)\leqslant C\psi(\lambda;t)$, Theorem \ref{additional} and condition \eqref{need-general-0} imply that
\[
\|w(t,\cdot)\|_{L^q(G)}\leqslant C\|w_0\|_{L^p(G)}\sup_{v>0}\big[\tau(E_{(0,v)}(|\mathscr{L}|))\big]^{\frac{1}{r}}\psi(t;v)\leqslant C\|w_0\|_{L^p(G)}<+\infty,
\]
completing the proof.
\end{proof}
Now we assume some restriction over the kernel $k$ to avoid putting the strong condition over the one-dimensional propagator, i.e. $s(t;\lambda)\leqslant C\psi(t;\lambda).$ Moreover, with these classes of kernels we can provide time-asymptotic behavior of equation \eqref{volterra-e-locally} which will depend on the same kernel.

\medskip The kernel $k\in\mathcal{PC}$ (\textit{completely positive}) if $k\in L_{loc}^1(\mathbb{R}_+)$ is nonnegative and nonincreasing, and there exists a nonnegative kernel $\mathscr{K}\in L_{loc}^1(\mathbb{R}_+)$ such that $\big(\mathscr{K}\ast k\big)(t)=1$ on $(0,+\infty),$ where $\ast$ is the Laplace convolution in \eqref{laplace}. The latter equality is so-called the Sonine condition \cite{sonine}. These classes were first introduced by Cl\'ement and Nohel in \cite{[50],Clement}. These kernels have been very useful in several fields. For instance, just to mention a few of them, in potential kernels \cite{[21]}, $p$-standard functions \cite{[198]}, non-local difussion equations \cite{uno2,Vergara1}, and the references therein. Different equivalent assertions can be given for this class, see e.g. \cite[Proposition 4.5]{pruss}, where Bernstein functions can be also used to describe them.  

\begin{thm}\label{integral-thm}
Let $G$ be a locally compact separable unimodular group and $1<p\leqslant 2\leqslant q<+\infty$. Let $\mathscr{L}$ be any left invariant operator on $G$. Assume that $k\in\mathcal{PC}$ is positive and  
\begin{equation}\label{need-general}
\sup_{t>0}\sup_{v>0}\big[\tau\big(E_{(0,v)}(|\mathscr{L}|)\big)\big]^{\frac{1}{p}-\frac{1}{q}}\frac{1}{1+v\int_0^t k(\tau){\rm d}\tau}<+\infty. 
\end{equation}
If $w_0\in L^p(G)$ then the solution of the integral equation \eqref{volterra-e-locally} is in $L^q(G)$.

Also, in the particular case that condition \eqref{trace-condition} holds, we get the following time decay rate for the solution of equation \eqref{volterra-e-locally} 
\[
\|w(t,\cdot)\|_{L^q(G)}\leqslant C_{\lambda,p,q}\bigg(\int_0^t k(\tau){\rm d}\tau\bigg)^{-\lambda\left(\frac{1}{p}-\frac{1}{q}\right)}\|w_0\|_{L^p(G)},\quad \frac{1}{\lambda}\geqslant\frac{1}{p}-\frac{1}{q}.    
\]
\end{thm}
\begin{proof}
The first part of the proof follows by Theorem \ref{integral-thm-0}. Indeed, it is enough to show that $s(t;\lambda)\leqslant C\big(1+\lambda(1\ast k)(t)\big)^{-1}.$ From equality \eqref{resolvent-e-1} we know that
\[
1=s(t;\lambda)+\lambda\int_0^t k(t-\tau)s(\tau;\lambda){\rm d}\tau,\quad t\geqslant 0. 
\]
By \cite[Proposition 2.1]{Clement} (see also \cite[Proposition 4.5, item (v)]{pruss}) we obtain that
\begin{align*}
    1\geqslant s(t;\lambda)+\lambda s(t;\lambda)\int_0^t k(t-\tau){\rm d}\tau \Longrightarrow s(t;\lambda)\leqslant \frac{1}{1+\lambda\int_0^t k(\tau){\rm d}\tau}.
\end{align*}
This proves our affirmation. Also, by the latter inequality, \eqref{volterra-sol}, \eqref{trace-condition} and Theorem \ref{additional} one obtains
\[
\|w(t,\cdot)\|_{L^q(G)}\leqslant C\|w_0\|_{L^p(G)}\sup_{v>0}\frac{v^{\frac{\lambda}{r}}}{1+v\int_0^t k(\tau){\rm d}\tau}.
\]
Notice first that for $\lambda/r=1$, the supremum is bounded by $\big((k\ast 1)(t)\big)^{-1}$. On the other hand, the above supremum is attained at $v=\frac{\lambda}{(r-\lambda)\int_0^t k(\tau){\rm d}\tau}$ whenever $\frac{1}{\lambda}>\frac{1}{p}-\frac{1}{q}$. Thus
\[
\|w(t,\cdot)\|_{L^q(G)}\leqslant  C_{\lambda,p,q}\bigg(\int_0^t k(\tau){\rm d}\tau\bigg)^{-\lambda/r}\|w_0\|_{L^p(G)}, 
\]
completing the proof.
\end{proof}

Let us now discuss a specific example of the above general results. Here we show the existence of the solution operator in the $L^p$ space. We consider the following integral equation over the space $L^p(\mathbb{G})$ $(1\leqslant p<+\infty$, $\mathbb{G}$ is a graded Lie group): 
\begin{equation}\label{integral-equivalent}
w(t,x)+\int_0^t l(s-t)\mathcal{R}w(s,x){\rm d}s=w_0(x),\quad t>0,\quad x\in\mathbb{G}, 
\end{equation}
where $w_0\in C(\mathbb{G})\cap L^p(\mathbb{G}),$ $l$ is a scalar kernel $\neq0$ in $L^1_{loc}(\mathbb{R}_+)$ and $\mathcal{R}$ is a positive (unbounded) Rockland operator  of homogeneous order $\nu$ on a graded group $\mathbb{G}$. 

Firstly, note that the Rockland operator $\mathcal{R}:\mathcal{D}(\mathbb{G})\subset L^p(\mathbb{G})\to L^p(\mathbb{G}),$ is densely defined in $L^p(\mathbb{G})$ whose domain $\mathcal{D}(\mathbb{G})$ is the space of smooth functions compactly supported in $\mathbb{G},$ see e.g. \cite[Subsection 4.3.1]{FR16} for more details and concretely \cite[Theorem 4.3.3]{FR16}.

By using \cite[Theorem 4.2]{pruss} and \cite[Corollary 4.2.9]{FR16}, it follows that the integral equation \eqref{integral-equivalent} admits a resolvent in $L^p(\mathbb{G})$ (exponentially bounded) whenever the initial condition $w_0$ is continuous in $L^p(\mathbb{G})$. Hence, we have that equation \eqref{integral-equivalent} is well-posed \cite[Proposition 1.1]{pruss}. 

Now we recall that \cite[Theorem 8.2]{david}: 
\[
\tau\big(E_{(0,v)}(\mathcal{R})\big)\lesssim v^{Q/\nu},\quad v\to+\infty.
\]

Therefore, as a consequence of Theorem \ref{integral-thm}, we can establish the following result.

\begin{cor}
Let $l\in \mathcal{PC}$ such that $l$ is positive and $l\in BV_{loc}(\mathbb{R}_+).$ Let $\mathbb{G}$ be a graded Lie group of homogeneous dimension $Q$. Let $\mathcal{R}$ be a positive Rockland operator of homogeneous degree $\nu$ on $\mathbb{G}$. If $w_0$ is a continuous function in $L^p(\mathbb{G})$ $(1\leqslant p<+\infty)$ then the integral equation \eqref{integral-equivalent} is well-posed. 
Moreover, we get the following time decay rate for the solution in $L^q(\mathbb{G})$ $(2\leqslant q<+\infty)$ with any datum in $L^p(\mathbb{G})$ $(1<p\leqslant 2):$ 
\begin{equation*}
\|w(t,\cdot)\|_{L^q(\mathbb{G})}\lesssim \bigg(\int_0^t l(\tau){\rm d}\tau\bigg)^{-Q/\nu\left(\frac{1}{p}-\frac{1}{q}\right)}\|w_0\|_{L^p(\mathbb{G})},\quad \frac{\nu}{Q}\geqslant\frac{1}{p}-\frac{1}{q}.    
\end{equation*}
\end{cor}
Also, it is worthy to mention that the explicit solution of equation \eqref{integral-equivalent} can be found using the Fourier analysis of the group, for more details see e.g. \cite{RRT}. 

\subsection{$\mathscr{L}$-evolutionary differential equations}

For $k\in \mathcal{PC}$, we study the following equation:
\begin{equation}\label{differential-lcg}
\begin{split}
\partial_t\big(k\ast[w(s,x)-w_0(x)]\big)(t)+\mathscr{L}w(t,x)&=0, \quad t>0,\quad x\in G, \\
w(t,\cdot)|_{_{_{t=0}}}&=w_0(\cdot)\in L^p(G),
\end{split}
\end{equation}
where $\mathscr{L}$ is a closed linear unbounded operator in $L^p(G)$ with dense domain $\mathcal{D}(\mathscr{L})$. Now, if we think about strong (differentiable) solution, we can then rewrite equation \eqref{differential-lcg} as
\[
\int_0^t k(t-s)\partial_s w(s,x){\rm d}s+\mathscr{L}w(t,x)=0.
\]
Let us do the convolution of the above equation with $\mathscr{K}$, and use the associativity of this operation along with $(\mathscr{K}\ast k)(t)=1$, then
\[
\int_0^t \partial_s w(s,x){\rm d}s+\mathscr{K}*\mathscr{L}w(t,x)=0,
\]
which implies 
\[
w(t,x)-w_0(x)+\int_0^t \mathscr{K}(t-s)\mathscr{L}w(s,x){\rm d}s=0.
\]
Hence, we arrive at the case of a general evolutionary integral equation, see Subsection \ref{integral-equation}. Now notice first that the kernel $\mathscr{K}\in\mathcal{PC}$ \cite[Theorem 2.2]{Clement}. Also, we need that the propagator $S(t)$ for the integral equation to be (at least) differentiable, then we have to assume additionally that $\mathscr{K}$ is positive and $\mathscr{K}\in BV_{loc}(\mathbb{R}_+)$ \cite[Proposition 1.2]{pruss}. So, we have the following assertion for the differential equation \eqref{differential-lcg}. We just write the result without proving it since it is very similar to the previous section. 

\begin{thm}\label{differential-equation}
Let $(k,\mathscr{K})\in \mathcal{PC}$ be such that $\mathscr{K}$ is positive and $\mathscr{K}\in BV_{loc}(\mathbb{R}_+).$ Let $G$ be a locally compact separable unimodular group and $1<p\leqslant 2\leqslant q<+\infty$. Let $\mathscr{L}$ be any positive left invariant operator on $G$. Suppose that 
\begin{equation*}
\sup_{t>0}\sup_{v>0}\big[\tau\big(E_{(0,v)}(|\mathscr{L}|)\big)\big]^{\frac{1}{p}-\frac{1}{q}}\frac{1}{1+v\int_0^t \mathscr{K}(\tau){\rm d}\tau}<+\infty. 
\end{equation*}
If $w_0\in L^p(G)$ then the solution of the partial integro-differential equation \eqref{differential-lcg} is in $L^q(G).$

Also, in the particular case that condition \eqref{trace-condition} holds, we get the following time decay rate for the solution of equation \eqref{differential-lcg} 
\[
\|w(t,\cdot)\|_{L^q(G)}\leqslant C_{\lambda,p,q}\bigg(\int_0^t \mathscr{K}(\tau){\rm d}\tau\bigg)^{-\lambda\left(\frac{1}{p}-\frac{1}{q}\right)}\|w_0\|_{L^p(G)},\quad \frac{1}{\lambda}\geqslant\frac{1}{p}-\frac{1}{q}.    
\]
\end{thm}

\subsection{Type of operators to get asymptotic decay}\label{type of operators}

In Theorems \ref{integral-thm} and \ref{differential-equation}, the time decay rate for the solution of equations \eqref{volterra-e-locally} and \eqref{differential-lcg} is predetermined by the condition \eqref{trace-condition}. Hence, let us mention briefly several examples of operators (in different groups) such that the trace of the spectral projections behave like $s^{\lambda}$ as $s\to+\infty$.

\begin{itemize}
\item[1.] For the Laplacian $\Delta_{\R^n}$ on the Euclidean space $\R^n$ we have \cite[Example 7.3]{RR2020}
\[
\tau\big(E_{(0,s)}(\Delta_{\R^n})\big)\lesssim s^{n/2},\quad s\to+\infty.
\]
\item[2.] The sub-Laplacian $\Delta_{sub}$ on a compact Lie group satisfies that  \cite{RR2020,[35]}
\[
\tau\big(E_{(0,s)}(-\Delta_{sub})\big)\lesssim s^{Q/2},\quad s\to+\infty,
\]
where $Q$ is the Hausdorff dimension of $G$ with respect to the control distance generated by the sub-Laplacian.

\item[3.] Let us consider the positive sub-Laplacian on the Heisenberg group $\mathbb{H}^n$. By \cite[Formula (7.17)]{RR2020}, it follows that
\[
\tau\big(E_{(0,s)}(\mathscr{L})\big)\lesssim s^{n+1},\quad s\to+\infty.
\]

\item[4.] For a positive Rockland operator $\mathcal{R}$ of order $\nu$ on a graded Lie group, we know \cite[Theorem 8.2]{david} that
\[
\tau\big(E_{(0,s)}(\mathcal{R})\big)\lesssim s^{Q/\nu},\quad s\to+\infty,
\]
where $Q$ is the homogeneous dimension of $G$. 

\item[5.] The non-Rockland-type operator $\mathfrak{D_1} = -(X_1^2 + X_2^2 + X_3^2 +X_4^2 +X_4^{-2})$ on the Engel group $\mathfrak{B}_4$, where $\{X_i\}$ are the vector fields that form the canonical basis of its Lie algebra. By \cite[Example 2.2]{Marianna} we get 
\[
\tau\big(E_{(0,s)}(\mathfrak{D_1})\big) \lesssim s^{3},\quad s\to+\infty.
\]

\item[6.] The non-Rockland-type operator $\mathfrak{D_2} = -(X_1^2 + X_2^2 + X_3^2 +X_4^2 + X_5^2+X_4^{-2}+X_5^{-2})$ on the Cartan group $\mathfrak{B}_5$, where $\{X_i\}$ are the vector fields that form the canonical basis of its Lie algebra. By \cite[Example 3.2]{Marianna}, we have
\[
\tau\big(E_{(0,s)}(\mathfrak{D_1})\big) \lesssim s^{9/2},\quad s\to+\infty.
\]

\item[7.] For an $m$-th order weighted subcoercive positive operator on a connected unimodular Lie group, one has  \cite[Proposition 0.3]{david2}
$$
\tau\left(E_{(0,s)}(\mathscr{L})\right) \lesssim s^{\frac{Q_*}{m}}, \quad s \rightarrow +\infty,
$$
where $Q_*$ is the local dimension of $G$ relative to the chosen weighted structure on its Lie algebra.
\item[8.] The Vladimirov operator $\mathfrak{D}^\mu$ $(\mu>0)$ \cite{[30]} on the group of $\rho$-adic numbers (abelian locally compact group denoted by $\Q_{\rho}$ with $\rho$ being a prime number) behaves like \cite[Section 4]{SRR}
\[
\tau\left(E_{(0,s)}\left(\mathfrak{D}^\mu\right)\right) = 1/\rho^v \lesssim s^{1/\mu},\quad s \rightarrow +\infty,
\]
where $v$ is the
smallest integer such that $1/\rho^{v}\sim s^{1/\mu}.$
\end{itemize}

\subsection{Examples}\label{Examples} Let us now discuss some applications of the results of the previous Subsections. We also recall, for the sake of completeness, some particular examples of integral equations of scalar type, which were already mentioned in \cite{RR2020} and \cite{SRR}. Everywhere below, we assume that $G$ is a locally compact separable unimodular group and $\mathscr{L}$ is any positive left invariant operator on $G$ such that condition \eqref{trace-condition} holds.  

Frequently, we will use the two-parametric Mittag-Leffler function:
\begin{equation}\label{bimittag}
E_{\alpha,\delta}(z)=\sum_{k=0}^{+\infty} \frac{z^k}{\Gamma(\alpha k+\delta)},\quad z,\delta\in\mathbb{C},\quad \Re(\alpha)>0,
\end{equation}
which is absolutely and locally uniformly convergent for the given parameters (\cite{mittag}). To estimate the propagators associated with these Mittag-Leffler functions, we recall the inequality \cite[Theorem 1.6]{page 35}: 
\begin{equation}\label{uniform-estimate}
|E_{\alpha,\delta}(z)|\leqslant \frac{C}{1+|z|},\quad z\in\mathbb{C},\quad \delta\in\mathbb{R},\quad\alpha<2,
\end{equation}
where $\mu\leqslant |\arg(z)|\leqslant \pi$, $\pi\alpha/2<\mu<\min\{\pi,\pi \alpha\}$ and $C$ is a positive constant.

In the following equations, we also use an integro-differential operator in time, the so-called Djrbashian--Caputo fractional derivative \cite{samko}. First, we recall the Sobolev spaces \cite[Appendix]{sobolev} which will be imposed over the function space of the operators: 
\[
W^{m,p}(I;X):=\left\{g\bigg/\exists\phi\in L^{p}(I;X): g(t)=\sum_{j=0}^{m-1}a_j \frac{t^{j}}{j!}+\frac{t^{m-1}}{(m-1)!}*\phi(t),\quad t\in I\right\},
\]
where $I$ is an interval in $\mathbb{R}$ and $X$ is a complex Banach space. One can see that $a_j=g^{(j)}(0)$ and $\phi(t)=g^{(j)}(t).$

Now we recall the Riemann--Liouville fractional integral of order $\beta>0$ (\cite{samko}) which is defined by
\[
\prescript{RL}{a}I^{\beta}f(t)=\frac1{\Gamma(\beta)}\int_a^t (t-s)^{\beta-1}f(s)\,\mathrm{d}s,\qquad f\in L^1(a,T),
\]
where $L^1(a,T)$ is the Lebesgue integrable space on $(a,T).$ Let us now introduce the Djrbashian--Caputo fractional derivative:
\[
\prescript{C}{a}D^{\beta}f(t)=\prescript{RL}{a}I^{n-\beta}f^{(n)}(t),\qquad f\in AC^n[a,T],\quad n=\lfloor\beta\rfloor+1,
\]
where $AC^n[a,T]$ is the set of functions such that $f^{(n-1)}$ exists and is absolutely continuous on $[a,T].$ The above operator is useful in applications since it can be rewritten utilizing the initial conditions as follows:
\begin{equation}\label{caputo-alternative-uso}
\prescript{C}{a}D^{\beta}f(t)=\prescript{RL}{a}D^{\beta}\left(f(t)-\sum_{k=0}^{n-1}\frac{f^{(k)}(a)}{k!}(t-a)^k\right),\qquad f\in AC^n[a,T],
\end{equation}
where $\prescript{RL}{a}D^{\beta}f(t)=D^{n}\prescript{RL}{a}I^{n-\beta}f(t)$ is the Riemann--Liouville fractional derivative. For more details of abstract fractional differential equations, see the works \cite{thesis2001,thesis,page 35,samko}.  

In our studies, we mean the next examples, we use the operator \eqref{caputo-alternative-uso} defined over the functions $f\in C^{m-1}(I)$ and $h_\beta*f\in W^{m,1}(I)$, where 
\[
h_\beta(t):=\left\{
\begin{array}{rccl}
\frac{t^{\beta-1}}{\Gamma(\beta)}, & t>0, \\
0,& t\leqslant0.
\end{array}
\right.
\]

\begin{ex}[$\mathscr{L}$-Heat equation] \label{heat-ex}
Let us consider the following $\mathscr{L}$-heat equation: 
\[
\partial_t w(t,x)+\mathscr{L}w(t,x)=0,\quad w(0,x)=w_0(x), \quad t>0,\quad x\in G.
\]
For $t>0$, we can apply the Borel functional calculus \cite{BorelFunctional} to obtain
\[
w(t,x)=e^{-t\mathscr{L}}w_0(x).
\]
So, $w$ satisfies the considered equation along with its initial condition. Therefore, by Corollary \ref{coro-lplq}, we get 
\[
\|w(t,\cdot)\|_{L^{q}(G)}\leqslant C_{\lambda,p,q}t^{-\lambda\left(\frac{1}{p}-\frac{1}{q}\right)}\|w_0\|_{L^p(G)}.
\]
\end{ex}

\begin{ex}[$\mathscr{L}$-Heat type equation]\label{HeatTypeEx}

We study the following heat type equation: 
\begin{equation}\label{heatlocally}
\begin{split}
^{C}\partial_{t}^{\beta}w(t,x)+\mathscr{L}w(t,x)&=0, \quad t>0,\quad x\in G,\quad 0<\beta<1, \\
w(t,x)|_{_{_{t=0}}}&=w_0(x).
\end{split}
\end{equation}
First, we have to note that the  pair $\big(t^{-\beta}/\Gamma(1-\beta),t^{\beta-1}/\Gamma(\beta)\big)\in\mathcal{PC}$ for $0<\beta<1.$ The solution operator of equation \eqref{heatlocally} is given by $w(t,x)=E_\beta(-t^\beta \mathscr{L})w_0(x)$ \cite[Chapter 3]{thesis} (see also \cite[Prop. 3.8 and Def. 2.3]{thesis2001}). From estimate \eqref{uniform-estimate} we have that $|E_{\beta}(-t^{\beta}\lambda)|\leqslant \frac{C}{1+t^{\beta}\lambda}\sim \frac{C}{1+\lambda(1\ast s^{\beta-1})(t)}.$ Thus, by Theorem \ref{differential-equation}, the solution $w$ is in $L^q(G)$ for the $\mathscr{L}$-heat type equation \eqref{heatlocally} whenever $w_0\in L^p(G).$ We also have that 
\[
\|w(t,\cdot)\|_{L^q(G)}\leqslant C_{\beta  ,\lambda,p,q}t^{-\beta\lambda\left(\frac{1}{p}-\frac{1}{q}\right)}\|w_0\|_{L^p(G)},\quad \frac{1}{\lambda}\geqslant \frac{1}{p}-\frac{1}{q},\quad t>0.    
\]
\end{ex}

\begin{ex}[$\mathscr{L}$-Wave type equation]\label{WaveTypeEx}

The following equation can  interpolate between wave (without being wave, $\beta<2$) and heat types:  
\begin{equation}\label{locallywave}
\begin{split}
^{C}\partial_{t}^{\beta}w(t,x)+\mathscr{L}w(t,x)&=0, \quad t>0,\quad x\in G,\quad 1<\beta<2, \\
w(t,x)|_{_{_{t=0}}}&=w_0(x), \\
\partial_t w(t,x)|_{_{_{t=0}}}&=w_1(x).
\end{split}
\end{equation}
The solution operator of equation \eqref{locallywave} is given by: 
\begin{equation}\label{wave-eq}
w(t,x)=E_\beta(-t^{\beta}\mathscr{L})w_0(x)+\prescript{RL}{0}I^{1}_t E_{\beta}(-t^{\beta}\mathscr{L})w_1(x).
\end{equation}
Moreover, it is easy to check that $\prescript{RL}{0}I^{1}_t E_{\beta}(-t^{\beta}s)=tE_{\beta,2}(-t^{\beta}s).$ So, by using the propagators of equation \eqref{wave-eq}, the last equality, the condition \eqref{need-general} for $\psi(t;v)=\frac{1}{1+t^{\beta}v}$, estimate \eqref{uniform-estimate} and Corollary \ref{coro-lplq}, we get 
\[
\|w(t,\cdot)\|_{L^q(G)}\lesssim t^{-\beta \lambda\left(\frac{1}{p}-\frac{1}{q}\right)}\big(\|w_0\|_{L^p(G)}+t\|w_1\|_{L^p(G)}\big).\]
\end{ex}

\begin{ex}[$\mathscr{L}$-Schr\"odinger type equation]

Consider the equation: 
\begin{align}\label{schro-ex}
\begin{split}
i\,^{C}\partial_{t}^{\beta}w(t,x)+\mathscr{L}w(t,x)&=0, \quad t>0,\quad x\in G,\quad 0<\beta<1, \\
w(t,x)|_{_{_{t=0}}}&=w_0(x).
\end{split}
\end{align}
 The solution operator of equation \eqref{schro-ex} is given by
\begin{equation*}
    w(t,x)=E_\beta(it^\beta \mathscr{L})w_0(x),\quad 0<\beta<1.
\end{equation*}  
Applying estimate \eqref{uniform-estimate} to the above propagator and Theorem \ref{differential-equation}, we get 
\[
\|w(t,\cdot)\|_{L^q(G)}\leqslant C_{\beta,\lambda,p,q}t^{-\beta \lambda \left(\frac{1}{p}-\frac{1}{q}\right)}\|w_0\|_{L^p(G)}.
\]
\end{ex}
By now, we are prepared to introduce some new type of equations which were not considered before nowhere in this setting. Of course, it is not just restricted to those ones, but it will give a wide panorama of generality and diversity of our results. 

\begin{ex}[Multi-term $\mathscr{L}$-heat type equations]
    
We consider the following multi-term heat type equation: 
\begin{equation}\label{multi-heatlocally}
\begin{split}
^{C}\partial_{t}^{\beta}w(t,x)+\sum_{k=1}^{m}\sigma_i \,^{C}\partial_{t}^{\beta_k}w(t,x)+\mathscr{L}w(t,x)&=0, \quad t>0,\quad x\in G, \\
w(t,x)|_{_{_{t=0}}}&=w_0(x),
\end{split}
\end{equation}
where $0<\beta_m<\cdots<\beta_1<\beta\leqslant 1$ and $\sigma_i>0$ for $i=1,\ldots,m.$ 

Notice that the kernel of the associated integral equation \eqref{multi-heatlocally} is given by $\mathfrak{K}(t)=t^{\beta-1}E_{(\beta-\beta_1,\ldots,\beta-\beta_m),\beta}(-\sigma_1 t^{\beta-\beta_1},\ldots,-\sigma_m t^{\beta-\beta_m})$ with the special property that $\mathfrak{K}(t)\in C(\mathbb{R}^+)\cap L^1_{loc}(\mathbb{R}^+)$ is a completely monotonic function, see page 98 and Theorem 3.2 of \cite{multi-cmf}. Therefore, by \cite[Corollary 2.4]{pruss}, we have that the integral equation associated to problem \eqref{multi-heatlocally} is well-posed and admits a bounded analytic solution operator $S(t)$. Thus, by Theorem \ref{differential-equation}, one has
\[
\|w(t,\cdot)\|_{L^q(G)}\leqslant C_{\beta,\beta_1,\ldots,\beta_m,\lambda,p,q}\bigg(\mathfrak{K}(t)\bigg)^{-\lambda \left(\frac{1}{p}-\frac{1}{q}\right)}\|w_0\|_{L^p(G)}.
\]
\end{ex}

\begin{ex}[$\mathscr{L}$-Cauchy problem with a time-variable coefficient] We study the equation:
\begin{equation}\label{variable-coe}
\begin{split}
\partial_t w(t)+\alpha(t)\mathscr{L}w(t)&=0, \quad t>0,\quad x\in G, \\
w(x,t)|_{_{_{t=0}}}&=w_0(x),
\end{split}
\end{equation}
where $\alpha:[0,+\infty)\to[0,+\infty)$ is a continuous function and $\mathscr{L}$ is the generator of a semigroup. The solution operator is given by:
\[
S(t)=\exp\left\{-\left(\int_0^t \alpha(s){\rm d}s\right)\mathscr{L}\right\},
\]
and by Corollary \ref{coro-lplq}, we have the following decay estimate:
\[
\|w(t,\cdot)\|_{L^q(G)}\leqslant C_{\lambda,p,q}\bigg(\int_0^t \alpha(s){\rm d}s\bigg)^{-\lambda\left(\frac{1}{p}-\frac{1}{q}\right)}\|w_0\|_{L^p(G)}.  
\]    
\end{ex}

\begin{ex}
Let us now analyze the homogeneous Rayleigh–Stokes problem for a generalized second-grade fluid by means of Riemann-Liouville fractional derivative. Some models can be found in e.g. \cite{apli0,apli5,apli28}. Here we consider the general case on a locally compact group $G.$ Thus, we consider the following problem: 
\begin{equation}\label{ray}
\begin{split}
\partial_t w(t,x)-(1+\gamma\prescript{RL}{}\partial_t^{\beta})\mathscr{L} w(t,x)&=0,\quad t>0,\quad x\in G, \\
w(t,x)|_{_{_{t=0}}}&=w_0(x),
\end{split}
\end{equation}
where $\gamma>0$ and $0<\beta<1.$ In these type of problems, the fractional derivative is somehow used to capture the viscoelastic behavior of the flow. Note that equation \eqref{ray} is equivalent to the following integral equation  
\[
w(t,x)=w_0(x)+\int_0^t \underbrace{\left(1+\gamma\frac{(t-s)^{-\beta}}{\Gamma(1-\beta)}\right)}_{k(t-s)}\mathscr{L}w(s){\rm d}s.
\]
Notice now that the kernel $k$ is in $L^1_{loc}(0,+\infty).$ Also, it is positive, decreasing and log $k$ convex such that $k(0^+)=+\infty.$ So, the Volterra equation $\big(\mathscr{K}\ast k\big)(t)=1$ for any $t>0$, has a unique solution \cite[Theorem 2.2 and Remark (iv)]{Clement} $\mathscr{K}\in L^1_{loc}(0,+\infty)$ which is nonnegative and nonincreasing. Thus, $(k,\mathscr{K})\in\mathcal{PC}$ and Theorem \ref{integral-thm} gives 
\begin{align*}
\|w(t,\cdot)\|_{L^q(G)}&\leqslant C_{\lambda,p,q}\bigg(\int_0^t k(\tau){\rm d}\tau\bigg)^{-\lambda\left(\frac{1}{p}-\frac{1}{q}\right)}\|w_0\|_{L^p(G)},\quad \frac{1}{\lambda}\geqslant\frac{1}{p}-\frac{1}{q}, \\    
&=C_{\lambda,p,q}\bigg(t+\frac{\gamma}{\Gamma(2-\beta)}t^{1-\beta}\bigg)^{-\lambda\left(\frac{1}{p}-\frac{1}{q}\right)}\|w_0\|_{L^p(G)}.
\end{align*}
\end{ex}

\begin{rem}
It is important to mention that for heat and wave type equations, we can recover the sharp estimate (time-decay)  given in \cite[Theorem 3.3, item (i)]{uno2} whenever $\frac{2}{n}>\frac{1}{p}-\frac{1}{q}$.
\end{rem}

\section{Well-posedness of nonlinear partial integro-differential equations}\label{strichartz}
In this section we study some type of linear and nonlinear integro-differential equations, specifically we aim for the local well-posedness\footnote{In the sense of \cite[Section 3.2]{tao}.} of such equations. A general equation of scalar type (see Section \ref{regularity}) is not treated since the generality did not allow us to manipulate or use at this moment the analysis which has been developed in previous sections for a wide class of equations predetermined by a random kernel. Nevertheless, we have the possibility to study some classical equations like the heat equation and the time-fractional versions of the heat and wave equation. First, we prove some time-space estimates of solutions to the corresponding non-homogeneous equations, then we define an appropriate Banach space to use some fix point type argument exploding such time-space estimates. Precisely, we are interested on controlling the following mixed-type norms: 
\begin{equation}
    \|w(t,x)\|_{L_t^r L_x^q([0,T] \times G)} :=\left\|\|w(t,\cdot)\|_{L_x^q(G)}\right\|_{L_t^r([0,T])}=\left(\int_{[0,T]} \|w(t,\cdot)\|_{L_x^q(G)}^r \,{\rm d}t\right)^{1/r}, 
\end{equation}
where $0<T< +\infty$ would be a fixed time. The parameters $q$ and $r$ will be fixed later.  

\medskip We begin this analysis with one of the most classic cases, i.e. the heat equation. Everywhere below, we assume that $G$ is a separable unimodular locally compact group and $\mathscr{L}$ is a positive left invariant operator acting on $G$.

\subsection{$\mathscr{L} $-Heat equation}
Let us start by considering the non-homogeneous $\mathscr{L}$-heat equation 
\begin{align}\label{NHHeat}
\begin{split}
\partial_{t}w(t,x)+\mathscr{L}w(t,x)&=f(t,x), \quad t>0,\quad x\in G, \\
w(t,x)|_{_{_{t=0}}}&=w_0(x),
\end{split}
\end{align}
whose solution is given by Duhamel's formula as: 
\[
w(t,x) = e^{-t\mathscr{L}}w_0(x) +\int_0^t e^{-(t-s)\mathscr{L}} f(s,x)\, {\rm d}s.
\]
 At this stage, for further discussions on the frame of exponents appearing on the mixed-type norms for which our solution will stay, we then introduce the concept of admisibility for a triple of exponents. This is closely related with the convergence of the solutions over the mixed-type norms.   
\begin{defn}\label{4.1}
    A triple $(r,q,p)$ is called $\mathscr{L}$-admissible if $1<p\leqslant 2\leqslant q <+\infty$, $1\leqslant r<+\infty$ and 
    \[
    \lambda\left(\frac{1}{p}-\frac{1}{q}\right) < \frac{1}{r},
    \]
    where $\lambda$ is the positive real number appearing in condition \eqref{trace-condition}. 
\end{defn}
\begin{rem}\label{remTriplesHeat}
Let us comment on the existence of such triples. Our main purpose is to prove well-posedness of nonlinear equations, so that $p$ will represent the regularity of the data, which would be fixed. Thus, let us describe two different  situations depending on the value of $p$. First of all, let us fix $1<p_0<2$. In this case, it may happen that there is no $q$ and $r$ such that $(r,q,p_0)$ is admissible. Indeed, if $\lambda\geqslant \frac{2p_0}{2-p_0}$ then it is impossible to find such triples. This is illustrated in Figure \ref{figura5}.
\begin{figure}[!ht]
\centering
\begin{tabular}{cc}
   \resizebox{0.45\textwidth}{!}{
\begin{circuitikz}
\tikzstyle{every node}=[font=\small]
\draw [->, >=Stealth] (0,0) .. controls (2,0) and (4,0) .. (6.3,0);
\draw [->, >=Stealth] (0,0) .. controls (0,5) and (0,5) .. (0,6.3);
\node [font=\normalsize] at (-1,6) {$\frac{1}{r}$};
\node [font=\small] at (6,-0.5) {$\frac{1}{q}$};
\node [font=\normalsize] at (-1,5) {$\frac{\lambda}{p_0}$};
\node [font=\small] at (-1,3.5) {$\lambda\frac{2-p_0}{2p_0}$};
\node [font=\small] at (-1,2.5) {$1$};
\node [font=\small] at (2.5,-0.5) {$1$};
\node [font=\small] at (1.25,-0.5) {$\frac{1}{2}$};
\node [font=\small] at (-1,-0.5) {$0$};
\node [font=\small, color=red] at (2.5,6) {$\lambda\geqslant\frac{2p_0}{2-p_0}$};

\draw [](2.5,0) to[short] (2.5,2.5);
\draw [](0,2.5) to[short] (2.5,2.5);
\draw [dashed, color={rgb,255:red,255; green,0; blue,0}, short] (0,5) to[] (1.25,3.5);
[
\end{circuitikz}
} & \resizebox{0.45\textwidth}{!}{
\begin{circuitikz}
\tikzstyle{every node}=[font=\small]
\draw [->, >=Stealth] (0,0) .. controls (2,0) and (4,0) .. (6.3,0);
\draw [->, >=Stealth] (0,0) .. controls (0,5) and (0,5) .. (0,6.3);
\node [font=\normalsize] at (-1,6) {$\frac{1}{r}$};
\node [font=\small] at (6,-0.5) {$\frac{1}{q}$};
\node [font=\normalsize] at (-1,3.5) {$\frac{\lambda}{p_0}$};
\node [font=\small] at (-1,1.5) {$\lambda\frac{2-p_0}{2p_0}$};
\node [font=\small] at (-1,2.5) {$1$};
\node [font=\small] at (2.5,-0.5) {$1$};
\node [font=\small] at (1.25,-0.5) {$\frac{1}{2}$};
\node [font=\small] at (-1,-0.5) {$0$};
\node [font=\small, color=blue] at (2.5,6) {$\lambda<\frac{2p_0}{2-p_0}$};

\draw [](2.5,0) to[short] (2.5,2.5);
\draw [](0,2.5) to[short] (2.5,2.5);
\draw [dashed, color={rgb,255:red,0; green,0; blue,255}, short] (0,3.5) to[] (1.25,1.5);
\fill[color={rgb,255:red,153; green,153; blue,255}]  (1.25,1.5) -- (1.25,2.5) -- (0.625,2.5) -- cycle;
[
\end{circuitikz}
} \\
\end{tabular}
\caption{Triples for $1<p_0<2$. In red, empty region of triples if $\lambda\geqslant\frac{2p_0}{2-p_0}$. In blue, existence of triples if $\lambda<\frac{2p_0}{2-p_0}$. }
\label{figura5}
\end{figure}
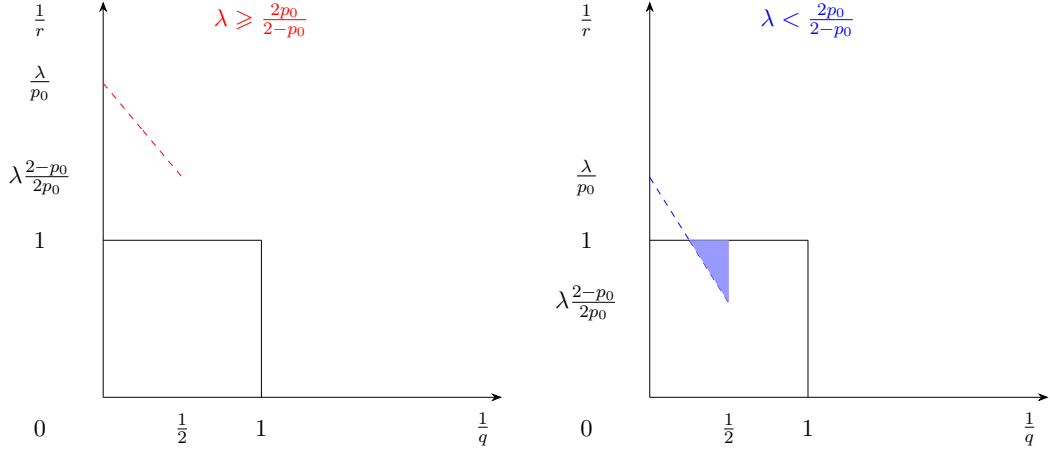

On the other hand, if we fix $p_0=2$ there will be always triples independently of the value of $\lambda$ as it is shown in Figure \ref{figura6}. Notice that there is more abundance of triples when $\lambda<2$. 
\begin{figure}[!ht]
\centering
\begin{tabular}{cc}
   \resizebox{0.45\textwidth}{!}{
\begin{circuitikz}
\tikzstyle{every node}=[font=\small]
\draw [->, >=Stealth] (0,0) .. controls (2,0) and (4,0) .. (6.3,0);
\draw [->, >=Stealth] (0,0) .. controls (0,5) and (0,5) .. (0,6.3);
\node [font=\normalsize] at (-1,6) {$\frac{1}{r}$};
\node [font=\small] at (6,-0.5) {$\frac{1}{q}$};
\node [font=\normalsize] at (-1,5) {$\frac{\lambda}{2}$};
\node [font=\small] at (-1,2.5) {$1$};
\node [font=\small] at (2.5,-0.5) {$1$};
\node [font=\small] at (1.25,-0.5) {$\frac{1}{2}$};
\node [font=\small] at (-1,-0.5) {$0$};
\node [font=\small, color=red] at (2.5,6) {$\lambda\geqslant2$};

\draw [](2.5,0) to[short] (2.5,2.5);
\draw [](0,2.5) to[short] (2.5,2.5);
\draw [dashed, color={rgb,255:red,255; green,0; blue,0}, short] (0,5) to[] (1.25,0);
\fill[color={rgb,255:red,255; green,153; blue,153}]  (1.25,0) -- (1.25,2.5) -- (0.625,2.5) -- cycle;
[
\end{circuitikz}
} & \resizebox{0.45\textwidth}{!}{
\begin{circuitikz}
\tikzstyle{every node}=[font=\small]
\draw [->, >=Stealth] (0,0) .. controls (2,0) and (4,0) .. (6.3,0);
\draw [->, >=Stealth] (0,0) .. controls (0,5) and (0,5) .. (0,6.3);
\node [font=\normalsize] at (-1,6) {$\frac{1}{r}$};
\node [font=\small] at (6,-0.5) {$\frac{1}{q}$};
\node [font=\small] at (-1,1.5) {$\frac{\lambda}{2}$};
\node [font=\small] at (-1,2.5) {$1$};
\node [font=\small] at (2.5,-0.5) {$1$};
\node [font=\small] at (1.25,-0.5) {$\frac{1}{2}$};
\node [font=\small] at (-1,-0.5) {$0$};
\node [font=\small, color=blue] at (2.5,6) {$\lambda<2$};

\draw [](2.5,0) to[short] (2.5,2.5);
\draw [](0,2.5) to[short] (2.5,2.5);
\draw [dashed, color={rgb,255:red,0; green,0; blue,255}, short] (0,1.5) to[] (1.25,0);
\draw [dashed, color={rgb,255:red,0; green,0; blue,255}, short] (0,1.5) to[] (0,2.5);
\fill[color={rgb,255:red,153; green,153; blue,255}]  (1.25,0) -- (1.25,2.5) -- (0,2.5) -- (0,1.5) -- cycle;
[
\end{circuitikz}
} \\
\end{tabular}
\caption{Triples for $p_0=2$. In red, a non-empty region of triples if $\lambda\geqslant2$. In blue, bigger region of triples if $\lambda<2$. }
\label{figura6}
\end{figure}
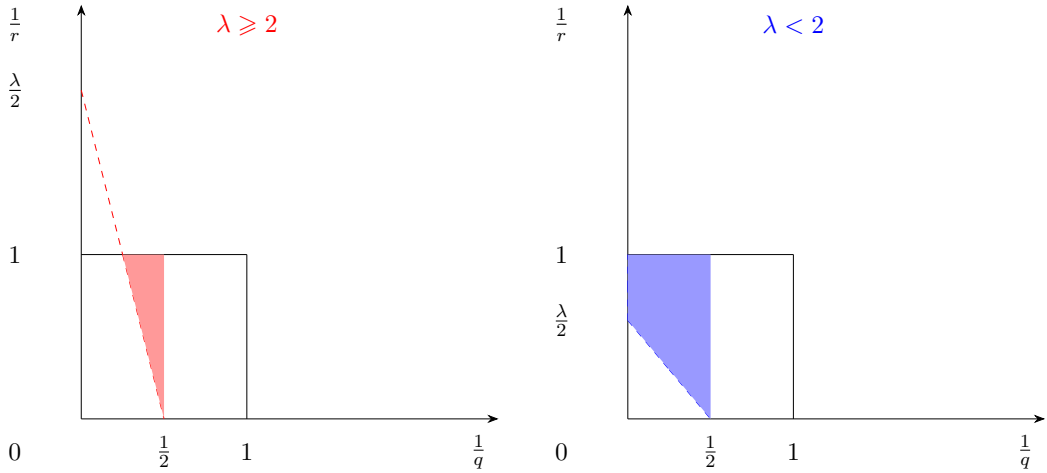
Henceforth, we will always work with non-empty triples. Thus it will be implicitly assumed that $\lambda<\frac{2p_0}{2-p_0}$ if $1<p_0<2$. 
\end{rem} 
Using Definition \ref{4.1} and Example \ref{heat-ex} we immediately obtain some space-time estimates for the homogeneous and non-homogenous part of the solution of equation \eqref{NHHeat} for finite time. 
\begin{prop}\label{timeSpaceEstimatesHeat}
    Suppose that the operator $\mathscr{L}$ satisfies the condition \eqref{trace-condition}. Let $(r,q,p)$ be a $\mathscr{L}$-admissible triple, $0<T<+\infty$ and $1\leqslant\rho\leqslant r <+\infty$. Then the solution of equation \eqref{NHHeat} satisfies 
    \begin{equation}
    \label{strichHeat}
        \|w\|_{L_t^r L_x^q([0,T] \times G)}\leqslant C_{\lambda,q,p,r,\rho,T}\left(\|w_0\|_{L_x^{p}(G)} + \|f\|_{L_t^{\rho} L_x^{p}([0,T] \times G)}\right),
    \end{equation}
    for some constant $C_{\lambda,q,p,r,\rho,T}>0$. 
\end{prop}
\begin{proof}
    On the one hand, utilizing Example \ref{heat-ex} we directly compute the mixed-type norm for the homogeneous part:
\begin{align*}
    \left\|e^{-t\mathscr{L}}w_0(x)\right\|_{L_t^r L_x^q([0,T]\times G)} &= \left\| \left\| e^{-t\mathscr{L}}w_0(x)\right\|_{L_x^q(G)}   \right\|_{L_t^r([0,T])}\\
    &\leqslant C_{\lambda,q,p}\left\|t^{-\lambda\left(\frac{1}{p}-\frac{1}{q}\right)}\|w_0\|_{L_x^{p}(G)}\right\|_{L_t^r([0,T])}\\
    &\leqslant C_{\lambda,q,p,r} T^{\frac{1}{r}-\lambda\left(\frac{1}{p}-\frac{1}{q}\right)}\|w_0\|_{L_x^{p}(G)},
\end{align*}
where the latter integral converges because $(r,q,p)$ is $\mathscr{L}$-admissible. On the other hand, again by Example \ref{heat-ex} and Young's inequality we get for the non-homogeneous part that
 \begin{align*}
        \left\|\int_0^t e^{-(t-s)\mathscr{L}} f(s,x)\, ds\right\|_{L_t^{r} L_x^{q}([0,T] \times G)} &\leqslant \left\|\int_0^t \left\|e^{-(t-s)\mathscr{L}} f(s,x)\right\|_{L^q(G)}\, {\rm d}s\right\|_{L_t^{r}([0,T])}\\
        &\leqslant C_{\lambda,q,p}\left\| t^{-\lambda\left(\frac{1}{p}-\frac{1}{q}\right)} * \left\|f(t,\cdot)\right\|_{L^{p}(G)}\right\|_{L_t^{r}([0,T])}\\
        &\leqslant C_{\lambda,q,p} \left\|t^{-\lambda\left(\frac{1}{p}-\frac{1}{q}\right)}\right\|_{L_t^\mu([0,T])}\|f\|_{L_t^{\rho}L_x^{p}([0,T] \times G)}
        \end{align*}
    where $\frac{1}{\mu}:=1+\frac{1}{r}-\frac{1}{\rho}$ comes out of Young's inequality. Since $(r,q,p)$ is $\mathscr{L}$-admissible and $1\leqslant\rho\leqslant r<+\infty$ we have that 
    \[
    \frac{1}{\mu}=\left(1-\frac{1}{\rho}\right) +\frac{1}{r}>\lambda\left(\frac{1}{p}-\frac{1}{q}\right) \text{ and } 1\leqslant \mu <+\infty,
    \]
    thus $(\mu,q,p)$ is $\mathscr{L}$-admissible and the $L_t^\mu$-norm converges. Precisely, we obtain that 
    \[
    \left\|\int_0^t e^{-(t-s)\mathscr{L}} f(s,x)\, ds\right\|_{L_t^{r} L_x^{q}([0,T] \times G)} \leqslant C_{\lambda,q,p,r,\rho} T^{\frac{1}{\mu}-\lambda\left(\frac{1}{p}-\frac{1}{q}\right)}\|f\|_{L_t^{\rho}L_x^{p}([0,T] \times G)},
    \]
    completing the proof.
\end{proof}
\begin{rem}\label{45}
Note that estimate \eqref{strichHeat} of Proposition \ref{timeSpaceEstimatesHeat} can be obtained in a similar way for a nonlinear function $F(t,w).$   \end{rem}
Now, let us consider the nonlinear $\mathscr{L}$-heat equation
\begin{align}\label{NLHeat}
\begin{split}
\partial_{t}w(t,x)+\mathscr{L}w(t,x)&=F(t,w) \quad t>0,\quad x\in G, \\
w(t,x)|_{_{_{t=0}}}&=w_0(x).
\end{split} 
\end{align} 
Below we denote by $C_t^0$ to be the space of continuous functions over $[0,T]$ equipped with the $L^\infty$ norm. 
\begin{defn}
    Let $1<p_0\leqslant 2$. We say that problem \eqref{NLHeat} is \textit{locally well-posed} in $L_x^{p_0}(G)$ if for any $w_0^*\in L_x^{p_0}(G)$ there exist a time $T$ and an open ball $B\subset L_x^{p_0}(G)$ containing $w_0^*$, and a subset $X$ of $C_t^0L_x^{p_0}([0,T]\times G)$, such that for any $w_0\in B$ there exists a strong unique solution\footnote{See \cite[Definition 3.4]{tao} for more details.} $w\in X$ to the integral equation 
    \[
    w(t,x) = e^{-t\mathscr{L}}w_0(x) +\int_0^t e^{-(t-s)\mathscr{L}}F(s,w)\,{\rm d}s,
    \]
    and furthermore the map $w_0\mapsto w$ is continuous from $B$ to $X$. 
\end{defn}
Using the triples and having in mind Proposition \ref{timeSpaceEstimatesHeat}, we define an appropriate Banach space in order to guarantee the well-posedness of equation \eqref{NLHeat}. 

Below we use the Schwartz--Bruhat spaces, which were introduced and developed by Bruhat \cite{bruhat} with the intention to have access to distribution theory in locally compact groups. All the details and properties can be found e.g. in \cite{bruhat}. These spaces are complete locally convex topological vector spaces that are continuously and densely contained in the space of compactly supported (continuous) functions. Moreover, they are dense in every $L^p(G)$, $1\leqslant p<+\infty$.

\medskip Let us fix $1< p_0\leqslant 2$, and let $S_{H}^{p_0}([0,T]\times G)$ be defined as the closure of the Schwartz-Bruhat functions under the norm 
\[
\|w\|_{S_{H}^{p_0}([0,T]\times G)} := \sup_{(r,q,p_0) \,\mathscr{L}\text{-admissible}}\|w\|_{C_t^0L_x^{p_0}([0,T]\times G)}+\|w\|_{L_t^r L_x^q([0,T]\times G)}. 
\] 
Hence, by using Remark \ref{45}, the inequality \eqref{strichHeat} (with a nonlinear function) becomes:
\begin{equation}\label{nonLinHeatEst}
    \|w(t,x)\|_{S_{H}^{p_0}([0,T]\times G)}\leqslant C_{\lambda,p_0,\rho,T}\left(\|w_0\|_{L_x^{p_0}(G)} + \|F(t,w)\|_{L_t^{\rho} L_x^{p_0}([0,T] \times G)}\right),
\end{equation}
for any $0<T<+\infty$, some $1\leqslant \rho<+\infty$ and a big enough constant $C_{\lambda,p_0,\rho,T}$. This constant will appear in the theorem and corollary below. 

Having obtained the time-space estimates and having defined the Banach space $S_{H}^{p_0}([0,T]\times G)$, we are in a position to prove a local well-posedness result in a very abstract set up. Remember that if we fix $1<p_0<2$, we are implicitly assuming that $\lambda<\frac{2p_0}{2-p_0}$ where $\lambda$ is the real number from the condition \eqref{trace-condition}.

\begin{thm}\label{ThNLH}
Let $1<p_0\leqslant 2$. Suppose that the operator $\mathscr{L}$ satisfies the condition \eqref{trace-condition}. Suppose that there exist $0<T<+\infty$ and $1\leqslant\rho<+\infty$ such that the nonlinearity $F$ satisfies the following estimate:
    \begin{equation}
    \label{nonLinConHeat}
        \|F(u) - F(v)\|_{L_t^{\rho} L_x^{p_0}([0,T] \times G)}\leqslant \frac{1}{2C_{\lambda,p_0,\rho,T}} \|u -v\|_{S_{H}^{p_0}([0,T]\times G)},
    \end{equation}
    for all $u,v\in B_{\varepsilon}:= \{u\in S_{H}^{p_0}([0,T]\times G): \|u\|\leqslant \varepsilon\}$, for some $\varepsilon>0$, where the constant $C_{\lambda,p_0,\rho,T}$ is from \eqref{nonLinHeatEst}. 
    Then the problem \eqref{NLHeat} is locally well-posed in $L_x^{p_0}(G)$. 
\end{thm}
\begin{proof}
     Let $T>0$ from the hypothesis. We are going to use \cite[Proposition 1.38]{tao}, so we set $\mathcal{S}=S_{H}^{p_0}([0,T]\times G)$ and $\mathcal{N}=L_t^{\rho} L_x^{p_0}([0,T] \times G)$. Let $B$ be a ball of fixed radius $R>0$ containing $w_0$, i.e. $\|w_0\|_{L_x^{p_0}(G)}\leqslant R$ and take $\varepsilon=2C_{\lambda,p_0,\rho,T} R$. Thus from the proof of Proposition \ref{timeSpaceEstimatesHeat}, we have that the homogeneous (linear) part of the solution satisfies 
    \[
    \|e^{-t\mathscr{L}}w_0\|_{S_{H}^{p_0}([0,T]\times G)}\leqslant \frac{\varepsilon}{2}. 
    \]
    Moreover, the same proof of Proposition \ref{timeSpaceEstimatesHeat} is also giving us that the nonlinear part of the solution satisfies 
    \[
    \left\|\int_0^t e^{-(t-s)\mathscr{L}}F(s,w)\,{\rm d}s\right\|_{S_{H}^{p_0}([0,T]\times G)}\leqslant C_{\lambda,p_0,\rho,T} \|F(t,w)\|_{L_t^{\rho} L_x^{p_0}([0,T] \times G)}.
    \]
    The latter inequalities together with the condition \eqref{nonLinConHeat} on $F$ are exactly the necessary conditions to apply the abstract iteration procedure \cite[Proposition 1.38]{tao}, therefore for $e^{-t\mathscr{L}}w_0\in B_{\varepsilon/2}$ there exists a unique solution $w\in B_\varepsilon$ to the problem \eqref{NLHeat} such that 
    \[
    \|w\|_{S_{H}^{p_0}([0,T]\times G)}\leqslant 2 \|e^{-t\mathscr{L}}w_0\|_{S_{H}^{p_0}([0,T]\times G)} \leqslant \varepsilon,
    \]
    completing the proof. 
\end{proof}
One can see that the condition \eqref{nonLinConHeat} imposed on the nonlinearity in Theorem \ref{ThNLH} is very abstract, so for completeness we provide explicitly an example of a function $F$ satisfying such condition.

\begin{cor}\label{coro-example-heat}
Let $1<p_0\leqslant 2$. Suppose that the operator $\mathscr{L}$ satisfies the condition \eqref{trace-condition}. Let $\eta$ be a $L_x^{p_0}(G)$ $\mathscr{L}$-heat subcritical exponent, i.e. $1<\eta<1+\frac{p_0}{\lambda}$, and let $\mu=\pm 1$. Then the problem $\eqref{NLHeat}$ is locally well-posed in $L_x^{p_0}(G)$ for $F(t,w)=\mu|w|^{\eta-1}w$. 
\end{cor}
\begin{proof}
    Let $T>0$ to be chosen later. We only need to prove condition \eqref{nonLinConHeat} for $F(t,w)=\mu|w|^{\eta-1}w$. The idea is to construct a convenient $\mathscr{L}$-admissible triple $(r,q, p_0)$ in order to take advantage of inequality \eqref{nonLinHeatEst} in the estimation procedure. We seek for numbers $r,q,\rho$ satisfying the following system of equations:
    \[
    r=\eta\rho,\hspace{0.5cm} \,q=\eta p_0, \hspace{0.5cm}\,\lambda\left(\frac{1}{p_0}-\frac{1}{q}\right)<\frac{1}{r}, \hspace{0.5cm}\,2\leqslant q<+\infty. 
    \]
    Since $1<\eta< 1+\frac{p_0}{\lambda}$ one can verify that by choosing $\rho$ such that 
    \[
    1\leqslant\rho <\frac{p_0}{\lambda(\eta-1)},
    \]
    it is possible to find $r$ and $q$ solving the previous system of equations, in particular turning $(r,q,p_0)$ into  an $\mathscr{L}$-admissible triple. We have to mention that in the case $1<p_0<2$, it is key to use the condition $\lambda<\frac{2p_0}{2-p_0}$ to guarantee the existence of such $\rho$. Hence, by construction (definition of $r$ and $q$) and the H\"older inequality, we get the following inequalities for all $u,v\in B_{\varepsilon}:= \{u\in S_{H}^{p_0}([0,T]\times G):\|u\|\leqslant \varepsilon\}$: 
    \begin{align*}
        \|\mu&|u|^{\eta-1}u - \mu|v|^{\eta-1}v\|_{L_t^{\rho} L_x^{p_0}([0,T] \times G)} \leqslant \\
        &\leqslant C_\eta \|(u-v)(|u|^{\eta-1} + |v|^{\eta-1})\|_{L_t^{\rho} L_x^{p_0}([0,T] \times G)}\\
        &\leqslant C_\eta \left\| \left[\left(\int_G |u-v|^{p_0\eta}\, dx\right)^{1/\eta} \left(\int_G ||u|^{\eta-1}+|v|^{\eta-1}|^{p_0\frac{\eta}{\eta-1}}\, dx\right)^{(\eta-1)/\eta}\right]^{1/p_0} \right\|_{L_t^{\rho}([0,T])}\\
        &\leqslant  C_\eta \left\| \|u-v\|_{L_x^q(G)}\||u|+|v|\|_{L_x^q(G)}^{\eta-1} \right\|_{L_t^{\rho}([0,T])} \\
        &\leqslant C_\eta \|u-v\|_{L_t^rL_x^q([0,T]\times G)}\||u|+|v|\|_{L_t^rL_x^q([0,T]\times G)}^{\eta-1} \,(\text{similar as previous H\"older})\\
        &\leqslant (2\varepsilon)^{\eta-1}C_{\eta}  \|u-v\|_{L_t^rL_x^q([0,T]\times G)}\\
        &\leqslant C_{\eta,\varepsilon}  \|u-v\|_{S^{p_0}_H([0,T] \times G)}, 
    \end{align*}
    where $\varepsilon=2C_{\lambda,p_0,\rho,T} R$ as in the proof of Theorem \ref{ThNLH}. Let us recall that according to Proposition \ref{timeSpaceEstimatesHeat} the dependence on $T$ of the constant $C_{\lambda,p_0,\rho,T}$ is given by $T^\xi$ for some $\xi>0$.  Hence, by choosing a small enough $T(p_0,\rho,\eta,\lambda,R)>0$ we are able to solve the equation 
    \[
        C_{\eta,\varepsilon} = \frac{1}{2C_{\lambda,p_0,\rho,T}}
    \]
    and the result follows. 
\end{proof}

\subsection{$\mathscr{L}$-Heat type equation}
Let us continue the analysis with case of $\mathscr{L}$-heat type equations. Consider the non-homogeneous equation
\begin{align}\label{asteriscoFH}
\begin{split}
\,^{C}\partial_{t}^{\beta}w(t,x)+\mathscr{L}w(t,x)&=f(t,x), \quad t>0,\quad x\in G,\quad 0<\beta<1, \\
w(t,x)|_{_{_{t=0}}}&=w_0(x).
\end{split}
\end{align}
From \cite{thesis2001} (see also  \cite{thesis}), we know that the solution to this equation is given by 
\begin{equation}
    w(t,x) = E_\beta(-t^\beta \mathscr{L})w_0(x)+ \int_0^t (t-s)^{\beta-1}E_{\beta,\beta}(-(t-s)^\beta\mathscr{L})f(s,x)\, {\rm d}s. 
\end{equation}
Notice that the estimations of the previous section did not explicitly covered the propagator $E_{\beta,\beta}(-t^\beta\mathscr{L})$ appearing in the non-homogeneous part of the solution, but this is not a problem since we can proceed as in Examples \ref{HeatTypeEx} and \ref{WaveTypeEx} to estimate its $L^q$ norm because, again, one has the estimation given by \eqref{uniform-estimate}: 
\[
|E_{\beta,\beta}(-t^\beta\lambda)|\leqslant \frac{C}{1+t^\beta\lambda}.
\]
\begin{lem}
\label{cotaOtroProp}
Let $0<\beta<1$, $1<p\leqslant 2\leqslant q<+\infty$ and $w\in L^p(G)$, then
    \[
    \|E_{\beta,\beta}(-t^\beta\mathscr{L})w(x)\|_{L^q(G)}\leqslant C_{\beta  ,\lambda,p,q}t^{-\beta\lambda\left(\frac{1}{p}-\frac{1}{q}\right)}\|w\|_{L^p(G)},\quad \frac{1}{\lambda}\geqslant \frac{1}{p}-\frac{1}{q},\quad t>0. 
    \]
\end{lem}
Having this inequality and the estimate for the homogeneous part (Example \ref{HeatTypeEx}), one just need to follow the steps we did for the heat equation, i.e., define some triples appropriately to obtain space-time estimates for the solution of  the equation \eqref{asteriscoFH}, and then use them to prove well-posedness of some nonlinear equations. 

\begin{defn}
    A triple $(r,q,p)$ is called $\beta$-$\mathscr{L}$-admissible if $1<p\leqslant 2\leqslant q <+\infty$, $1\leqslant r< +\infty$ and 
    \[
    \beta\lambda\left(\frac{1}{p}-\frac{1}{q}\right) < \frac{1}{r},
    \]
    where $\lambda$ is the real number appearing in condition \eqref{trace-condition}. 
\end{defn}
\begin{rem}
    Note that as in the case of the heat equation, one could have cases of empty regions of triples but we omit such analysis since it is very similar to the one on Remark \ref{remTriplesHeat}. Below we always assume that we are dealing with non-empty triples. 
\end{rem}
\begin{prop}\label{timeSpaceEstimatesHeatFrac}
    Suppose that the operator $\mathscr{L}$ satisfies the condition \eqref{trace-condition}. Let $0<T<+\infty$, let $(r,q,p)$ be a $\beta$-$\mathscr{L}$-admissible triple such that $\frac{1}{\lambda}\geqslant\frac{1}{p}-\frac{1}{q},$ and let $\frac{1}{\beta}\leqslant \rho\leqslant r<+\infty$. Then the solution of equation \eqref{asteriscoFH} satisfies 
    \begin{equation}
    \label{strichHeatT}
        \|w\|_{L_t^r L_x^q([0,T] \times G)}\leqslant C_{\beta,\lambda,q,p,r,\rho,T}\left(\|w_0\|_{L_x^{p}(G)} + \|f\|_{L_t^{\rho} L_x^{p}([0,T] \times G)}\right),
    \end{equation}
    for some $C_{\beta, \lambda,q,p,r,\rho,T}>0$. 
\end{prop}
\begin{proof}
    For the homogenoeus part we can directly compute the mixed-type norm utilizing Example \ref{HeatTypeEx}. Indeed
\begin{align*}
    \left\|E_\beta(-t^\beta \mathscr{L})w_0(x)\right\|_{L_t^r L_x^q([0,T]\times G)} &= \left\| \left\| E_\beta(-t^\beta \mathscr{L})w_0(x)\right\|_{L_x^q(G)}   \right\|_{L_t^r([0,T])}\\
    &\leqslant C_{\beta,\lambda,q,p}\left\|t^{-\beta\lambda\left(\frac{1}{p}-\frac{1}{q}\right)}\|w_0\|_{L_x^{p}(G)}\right\|_{L_t^r([0,T])}\\
    &\leqslant C_{\beta,\lambda,q,p,r} T^{\frac{1}{r}-\beta\lambda\left(\frac{1}{p}-\frac{1}{q}\right)}\|w_0\|_{L_x^{p}(G)},
\end{align*}
where the latter integral converges because $(r,q,p)$ is $\beta$-$\mathscr{L}$-admissible. For the non-homogeneous part we carry on the estimation process using Lemma \ref{cotaOtroProp} and Young's inequality, therefore we get
 \begin{align*}
        \left\|\int_0^t (t-s)^{\beta-1}E_{\beta,\beta}(-\right.&(t-s)^\beta\mathscr{L})f(s,x)\, {\rm d}s\bigg\|_{L_t^{r} L_x^{q}([0,T] \times G)} \\
        &\leqslant \left\|\int_0^t (t-s)^{\beta-1}\left\|E_{\beta,\beta}(-(t-s)^\beta\mathscr{L}) f(s,x)\right\|_{L^q(G)}\, {\rm d}s\right\|_{L_t^{r}([0,T])}\\
        &\leqslant C_{\beta,\lambda,q,p}\left\| t^{-\left(1-\beta +\beta\lambda\left(\frac{1}{p}-\frac{1}{q}\right)\right)} * \left\|f(t,\cdot)\right\|_{L^{p}(G)}\right\|_{L_t^{r}([0,T])}\\
        &\leqslant C_{\beta,\lambda,q,p} \left\|t^{-\left(1-\beta +\beta\lambda\left(\frac{1}{p}-\frac{1}{q}\right)\right)}\right\|_{L_t^\mu([0,T])}\|f\|_{L_t^{\rho}L_x^{p}([0,T] \times G)}\\
        &\leqslant C_{\beta,\lambda,q,p,r,\rho} T^{\frac{1}{\mu}-\left(1-\beta +\beta\lambda\left(\frac{1}{p}-\frac{1}{q}\right)\right)}\|f\|_{L_t^{\rho}L_x^{p}([0,T] \times G)},
    \end{align*}
    where $\frac{1}{\mu}:=1+\frac{1}{r}-\frac{1}{\rho}>1-\beta +\beta\lambda\left(\frac{1}{p}-\frac{1}{q}\right)$ and $1\leqslant \mu <+\infty$ since $(r,q,p)$ is $\beta$-$\mathscr{L}$-admissible and $\frac{1}{\beta}\leqslant\rho\leqslant r <+\infty$, which guarantees convergence of last integral on time.
\end{proof}

\begin{rem}
Note that in Propositions \ref{timeSpaceEstimatesHeat} and \ref{timeSpaceEstimatesHeatFrac}, the parameter $\rho$ is changing due to the nature of the classical heat equation and the counterpart of the fractional one. In fact, in Proposition \ref{timeSpaceEstimatesHeat}, for the case of the classical heat equation, we get $1\leqslant \rho\leqslant r <+\infty.$ While, for the fractional heat equation, in Proposition \ref{timeSpaceEstimatesHeatFrac}, the order $\beta$ of the time-fractional derivative appears somehow as a restriction in the condition $1<\frac{1}{\beta}\leqslant\rho\leqslant r <+\infty.$
\end{rem}

We move on to the study of the nonlinear $\mathscr{L}$-heat type equation
\begin{align}\label{NLHeatT}
\begin{split}
\,^{C}\partial_{t}^{\beta}w(t,x)+\mathscr{L}w(t,x)&=F(t,w), \quad t>0,\quad x\in G,\quad 0<\beta<1, \\
w(t,x)|_{_{_{t=0}}}&=w_0(x).
\end{split}
\end{align} 
The definition of well-posedness and adequate Banach space are made mutatis mutandis the ones of the classical heat equation treated before.  
\begin{defn}
Let $1<p_0\leqslant 2$. We say that the problem \eqref{NLHeatT} is \textit{locally well-posed} in $L_x^{p_0}(G)$ if for any $w_0^*\in L_x^{p_0}(G)$ there exist a time $T$ and an open ball $B\subset L_x^{p_0}(G)$ containing $w_0^*$, and a subset $X$ of $C_t^0L_x^{p_0}([0,T]\times G)$, such that for any $w_0\in B$ there exists a strong unique solution $w\in X$ to the integral equation 
    \[
    w(t) =  E_\beta(-t^\beta \mathscr{L})w_0(x)+ \int_0^t (t-s)^{\beta-1}E_{\beta,\beta}(-(t-s)^\beta\mathscr{L})F(s,w)\, {\rm d}s,
    \]
    and furthermore the map $w_0\mapsto w$ is continuous from $B$ to $X$. 
\end{defn}
Let us fix $1< p_0\leqslant 2$, and let $S_{HT}^{p_0}([0,T]\times G)$ be defined as the closure of the Schwartz-Bruhat functions under the norm 
\[
\|w\|_{S_{HT}^{p_0}([0,T]\times G)} := \sup_{(r,q,p_0) \,\beta\text{-}\mathscr{L}\text{-admissible}} \|w\|_{C_t^0L_x^{p_0}([0,T]\times G)}+\|w\|_{L_t^r L_x^q([0,T]\times G)}. 
\] 
We omit the proof of the following theorem since it is essentially the same as the one of Theorem \ref{ThNLH}. Notice that for this proof is necessary to use the inequality \eqref{strichHeatT} (with a non-linear term $F$), i.e.:
\begin{equation}\label{nonLinHeatEstT-00}
    \|w(t,x)\|_{S_{HT}^{p_0}([0,T]\times G)}\leqslant C_{\beta,\lambda,p_0,\rho,T}\left(\|w_0\|_{L_x^{p_0}(G)} + \|F(t,w)\|_{L_t^{\rho} L_x^{p_0}([0,T] \times G)}\right), 
\end{equation}
for any $0<T<+\infty$, some $\frac{1}{\beta}\leqslant\rho<+\infty$ and some big enough constant $C_{\beta,\lambda,p_0,\rho,T}$. Notice that such constant will appear in the following theorem and corollary (this is different from the constant $C_{\lambda,p_0,\rho,T}$ which appeared previously in the heat equation).

\begin{thm}\label{NLFHE}
Let $1<p_0\leqslant 2$. Suppose that the operator $\mathscr{L}$ satisfies the condition \eqref{trace-condition}. Suppose that there exist $0<T<+\infty$ and $\frac{1}{\beta}\leqslant\rho<+\infty$ such that the nonlinearity $F$ satisfies the following estimate
    \begin{equation}
    \label{nonLinConHeatT}
        \|F(u) - F(v)\|_{L_t^{\rho} L_x^{p_0}([0,T] \times G)}\leqslant\frac{1}{2C_{\beta,\lambda,p_0,\rho,T}} \|u -v\|_{S_{HT}^{p_0}([0,T]\times G)}
    \end{equation}
    for all $u,v\in B_{\varepsilon}:= \{u\in S_{HT}^{p_0}([0,T]\times G): \|u\|\leqslant \varepsilon\}$, for some $\varepsilon>0$, where the constant $C_{\beta,\lambda,p_0,\rho,T}$ is from \eqref{nonLinHeatEstT-00}. 
    Then the problem \eqref{NLHeatT} is locally well-posed in $L_x^{p_0}(G)$. 
\end{thm}
As in the case of the classical heat equation we can provide a concrete example of a nonlinearity of polynomial type. It is important to highlight that the proof is similar to the one of Corollary \ref{coro-example-heat}. The main difference is the admisibility triple that plays a crucial role in the convergence of the considered norms. Hence, we include some of the calculations and details of its proof.
\begin{cor}
Let $1<p_0\leqslant 2$. Suppose that the operator $\mathscr{L}$ satisfies the condition \eqref{trace-condition}. Let $\eta$ be a $L_x^{p_0}(G)$ $\mathscr{L}$-heat subcritical exponent $\left(1<\eta< 1+\frac{p_0}{\lambda}\right)$, and let $\mu=\pm 1$. Then the problem $\eqref{NLHeatT}$ is locally well-posed in $L_x^{p_0}(G)$ for $F(t, w)=\mu|w|^{\eta-1}w$. 
\end{cor}
\begin{proof}
Let $T>0$ to be chosen later. It is sufficient to guarantee condition \eqref{nonLinConHeatT} for $F(t,w)=\mu|w|^{\eta-1}w$. So, in this case, we look for numbers $r,q,\rho$ such that:
\[
    r=\eta\rho,\hspace{0.5cm} \,q=\eta p_0, \hspace{0.5cm}\,\beta\lambda\left(\frac{1}{p_0}-\frac{1}{q}\right)<\frac{1}{r}, \hspace{0.5cm}\,2\leqslant q<+\infty. 
\]
Using the condition that $1<\eta< 1+\frac{p_0}{\lambda}$ we can choose  $\rho$ in the following frame 
    \[
    \frac{1}{\beta}\leqslant\rho < \frac{p_0}{\beta\lambda(\eta-1)}.
    \]
Therefore, we can find $r$ and $q$ solving the previous system of equations. Also, the triple $(r,q,p_0)$ is $\beta$-$\mathscr{L}$-admissible. Once we have the desired triple $(r,q,p_0)$ we can estimate $\|F(u) - F(v)\|_{L_t^{\rho} L_x^{p_0}([0,T] \times G)}$ in the same way as in Corollary \ref{coro-example-heat}, that yields to the verification of the abstract condition on $F$:
    \begin{align*}
        \|\mu|u|^{\eta-1}u - \mu|v|^{\eta-1}v\|_{L_t^{\rho} L_x^{p_0}([0,T] \times G)} \leqslant C_{\eta,\varepsilon}  \|u-v\|_{S^{p_0}_{HT}([0,T] \times G)}, 
    \end{align*}
    where $\varepsilon=2C_{\beta, \lambda,p_0,\rho,T} R$ would come from the proof of Theorem \ref{NLFHE}. Again the constant $C_{\beta, \lambda,p_0,\rho,T}$ depends on $T^\xi$ for some $\xi>0$, so by choosing a small enough $T(\beta, p_0,\rho,\eta,\lambda,R)>0$  we conclude the proof. 
\end{proof}

\subsection{$\mathscr{L}$-Wave type equation}
We now analyze the case of non-homogeneous $\mathscr{L}$-wave type equation
\begin{align}\label{asterisco-wave}
\begin{split}
\,^{C}\partial_{t}^{\beta}w(t,x)+\mathscr{L}w(t,x)&=f(t,x), \quad t>0,\quad x\in G,\quad 1<\beta<2, \\
w(t,x)|_{_{_{t=0}}}&=w_0(x), \\
\partial_t w(t,x)|_{_{_{t=0}}}&=w_1(x).
\end{split}
\end{align}
The solution is given by 
\begin{align*}
    w(t,x)=E_{\beta}(-t^{\beta}\mathscr{L})w_0(x)+&\prescript{RL}{0}I^{1}_{t} E_{\beta}(-t^{\beta}\mathscr{L}) w_1(x) \nonumber \\
    &+  \int_0^t (t-s)^{\beta-1}E_{\beta,\beta}(-(t-s)^\beta\mathscr{L})f(s,x)\, {\rm d}s.
\end{align*}
Notice that we can control the propagator appearing in the non-homogenoeus part of the solution in the same way as in Lemma \ref{cotaOtroProp} since the estimate \eqref{uniform-estimate} holds as well for $1<\beta<2$, thus we have:
\begin{lem}
\label{cotaOtroProp-w}
Let $1<\beta<2$, $1<p\leqslant 2\leqslant q<+\infty$ and $w\in L^p(G)$, then
    \[
    \|E_{\beta,\beta}(-t^\beta\mathscr{L})w(x)\|_{L^q(G)}\leqslant C_{\beta  ,\lambda,p,q}t^{-\beta\lambda\left(\frac{1}{p}-\frac{1}{q}\right)}\|w\|_{L^p(G)},\quad \frac{1}{\lambda}\geqslant \frac{1}{p}-\frac{1}{q},\quad t>0. 
    \]
\end{lem}
By Lemma \ref{cotaOtroProp-w} and the estimate for the homogeneous part (Example \ref{locallywave}), again we follow the recipe established in the heat equation case. Hence, we first give the following definition. 
\begin{defn}
    A triple $(r,q,p)$ is called $\beta_w$-$\mathscr{L}$-admissible if $1<p\leqslant 2\leqslant q <+\infty$, $1\leqslant r< 2$,  
    \[
    \beta\lambda\left(\frac{1}{p}-\frac{1}{q}\right) < \frac{1}{r} \,\text{ and } \,1-\beta\lambda\left(\frac{1}{p}-\frac{1}{q}\right) < \frac{1}{r},
    \]
    where $\lambda$ is the real number appearing in condition \eqref{trace-condition}. 
\end{defn}
\begin{rem}
We point out two important differences with the triples occurring in the cases of the heat and heat type equations. On the one hand, there is an extra condition $1-\beta\lambda\left(\frac{1}{p}-\frac{1}{q}\right) < \frac{1}{r}$, coming from the second initial condition, that is restricting even more the existence of such triples. On the other hand, those two conditions together also restrict dramatically the values of the $L_t^r$ regularity since $1\leqslant r <2$. Nevertheless, this does not preclude the existence of such triples, in fact in the following figures (exactly Figure \ref{figura7}) we illustrate how triples always exist in the case $p_0=2$:
    \begin{figure}[!ht]
\centering
\begin{tabular}{cc}
   \resizebox{0.45\textwidth}{!}{
\begin{circuitikz}
\tikzstyle{every node}=[font=\small]
\draw [->, >=Stealth] (0,0) .. controls (2,0) and (4,0) .. (6.3,0);
\draw [->, >=Stealth] (0,0) .. controls (0,5) and (0,5) .. (0,6.3);
\node [font=\normalsize] at (-1,6) {$\frac{1}{r}$};
\node [font=\small] at (6,-0.5) {$\frac{1}{q}$};
\node [font=\normalsize] at (-1,3.75) {$\frac{\beta\lambda}{2}$};
\node [font=\small] at (-1,2.5) {$1$};
\node [font=\small] at (2.5,-0.5) {$1$};
\node [font=\small] at (1.25,-0.5) {$\frac{1}{2}$};
\node [font=\small] at (-1,-0.5) {$0$};
\node [font=\small, color=red] at (2.5,6) {$\beta\lambda>1$};

\draw [](2.5,0) to[short] (2.5,2.5);
\draw [](0,2.5) to[short] (2.5,2.5);
\draw [dashed, color={rgb,255:red,255; green,0; blue,0}, short] (0,3.75) to[] (1.25,0);
\draw [dashed, color={rgb,255:red,255; green,0; blue,0}, short] (0,-1.25) to[] (1.25,2.5);
\draw [dashed, color={rgb,255:red,255; green,0; blue,0}, short] (0,1.25) to[] (1.25, 1.25);
\fill[color={rgb,255:red,255; green,153; blue,153}]  (0.417,2.5) -- (1.25,2.5) -- (0.834,1.25) -- cycle;
[
\end{circuitikz}
} & \resizebox{0.45\textwidth}{!}{
\begin{circuitikz}
\tikzstyle{every node}=[font=\small]
\draw [->, >=Stealth] (0,0) .. controls (2,0) and (4,0) .. (6.3,0);
\draw [->, >=Stealth] (0,0) .. controls (0,5) and (0,5) .. (0,6.3);
\node [font=\normalsize] at (-1,6) {$\frac{1}{r}$};
\node [font=\small] at (6,-0.5) {$\frac{1}{q}$};
\node [font=\small] at (-1,1) {$\frac{\beta\lambda}{2}$};
\node [font=\small] at (-1,2.5) {$1$};
\node [font=\small] at (2.5,-0.5) {$1$};
\node [font=\small] at (1.25,-0.5) {$\frac{1}{2}$};
\node [font=\small] at (-1,-0.5) {$0$};
\node [font=\small, color=blue] at (2.5,6) {$\beta\lambda\leqslant 1$};

\draw [](2.5,0) to[short] (2.5,2.5);
\draw [](0,2.5) to[short] (2.5,2.5);
\draw [dashed, color={rgb,255:red,0; green,0; blue,255}, short] (0,1) to[] (1.25,0);
\draw [dashed, color={rgb,255:red,0; green,0; blue,255}, short] (0,1.5) to[] (1.25,2.5);
\draw [dashed, color={rgb,255:red,0; green,0; blue,255}, short] (0,1.25) to[] (1.25,1.25);
\draw [dashed, color={rgb,255:red,0; green,0; blue,255}, short] (0,1.25) to[] (0,2.5);
\fill[color={rgb,255:red,153; green,153; blue,255}]  (0,1.5) -- (0,2.5) --  (1.25,2.5) -- cycle;
[
\end{circuitikz}
} \\
\end{tabular}
\caption{Triples for $p_0=2$. In red, a non-empty region of triples if $\beta\lambda>1$. In blue, a different non-empty region of triples if $\beta\lambda<1$. }
\label{figura7}
\end{figure}
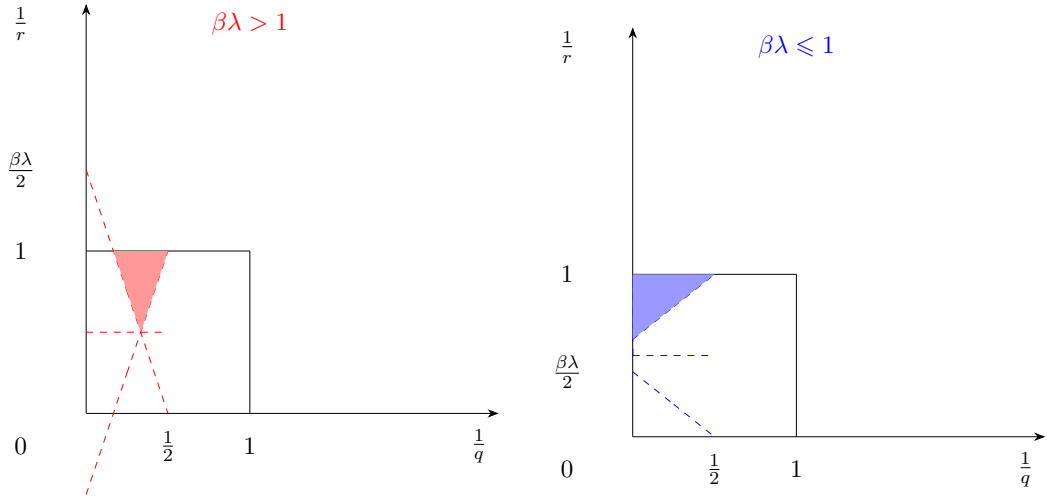
\end{rem}
\begin{prop}\label{timeSpaceEstimatesHeatFrac-w}
    Suppose that the operator $\mathscr{L}$ satisfies the condition \eqref{trace-condition}. Let $0<T<+\infty$, let $(r,q,p)$ be a $\beta_w$-$\mathscr{L}$-admissible triple such that $\frac{1}{\lambda}\geqslant\frac{1}{p}-\frac{1}{q}$,
    and let $1\leqslant \rho\leqslant r < 2$. Then the solution of the equation \eqref{asterisco-wave} satisfies 
    \begin{equation}
    \label{strichWave}
        \|w\|_{L_t^r L_x^q([0,T] \times G)}\leqslant C_{\beta,\lambda,q,p,r,\rho,T}\left(\|w_0\|_{L_x^{p}(G)} + \|w_1\|_{L_x^{p}(G)}+ \|f\|_{L_t^{\rho} L_x^{p}([0,T] \times G)}\right)
    \end{equation}
    for some $C_{\beta,\lambda,q,p,r,\rho,T}>0$. 
\end{prop}
\begin{proof}
By using Example \ref{HeatTypeEx}, we can estimate the homogeneous part as follows: 
\begin{align*}
    \big\|E_{\beta}(-t^{\beta}\mathscr{L})&w_0(x)+\prescript{RL}{0}I^{1}_{t} E_{\beta}(-t^{\beta}\mathscr{L}) w_1(x)\big\|_{L_t^r L_x^q([0,T]\times G)} =\\ 
    &\left\| \left\|E_{\beta}(-t^{\beta}\mathscr{L})w_0(x)+\prescript{RL}{0}I^{1}_{t} E_{\beta}(-t^{\beta}\mathscr{L}) w_1(x)\right\|_{L_x^q(G)}   \right\|_{L_t^r([0,T])}\\
    &\leqslant C_{\beta,\lambda,q,p}\left\|t^{-\beta \lambda\left(\frac{1}{p}-\frac{1}{q}\right)}\left(\|w_0\|_{L^p(G)}+t\|w_1\|_{L^p(G)}\right)\right\|_{L_t^r([0,T])}\\
    &\leqslant C_{\beta,\lambda,q,p} \left\|t^{-\beta \lambda\left(\frac{1}{p}-\frac{1}{q}\right)} \|w_0\|_{L^p(G)}\right\|_{L_t^r([0,T])} + \left\|t^{1-\beta \lambda\left(\frac{1}{p}-\frac{1}{q}\right)} \|w_1\|_{L^p(G)}\right\|_{L_t^r([0,T])}\\
    &\leqslant C_{\beta,\lambda,q,p,r,T} \left(\|w_0\|_{L_x^{p}(G)}+\|w_1\|_{L_x^{p}(G)}\right),
\end{align*}
where the latter integrals converge since $(r,q,p)$ is $\beta_w$-$\mathscr{L}$-admissible. For the non-homogeneous part we apply Lemma \ref{cotaOtroProp-w} and Young's inequality, hence
 \begin{align*}
        \left\|\int_0^t (t-s)^{\beta-1}E_{\beta,\beta}(-\right.&(t-s)^\beta\mathscr{L})f(s,x)\, {\rm d}s\bigg\|_{L_t^{r} L_x^{q}([0,T] \times G)} \\
        &\leqslant \left\|\int_0^t (t-s)^{\beta-1}\left\|E_{\beta,\beta}(-(t-s)^\beta\mathscr{L}) f(s,x)\right\|_{L^q(G)}\, {\rm d}s\right\|_{L_t^{r}([0,T])}\\
        &\leqslant C_{\beta,\lambda,q,p}\left\| t^{-(1-\beta +\beta\lambda\left(\frac{1}{p}-\frac{1}{q}\right))} * \left\|f(t,\cdot)\right\|_{L^{p}(G)}\right\|_{L_t^{r}([0,T])}\\
        &\leqslant C_{\beta,\lambda,q,p} \left\|t^{-\left(1-\beta +\beta\lambda\left(\frac{1}{p}-\frac{1}{q}\right)\right)}\right\|_{L_t^\mu([0,T])}\|f\|_{L_t^{\rho}L_x^{p}([0,T] \times G)}\\
        &\leqslant C_{\beta,\lambda,q,p,r,\rho} T^{\frac{1}{\mu}-\left(1-\beta +\beta\lambda\left(\frac{1}{p}-\frac{1}{q}\right)\right)}\|f\|_{L_t^{\rho}L_x^{p}(I \times G)},
    \end{align*}
    where $\frac{1}{\mu}:=1+\frac{1}{r}-\frac{1}{\rho}>1-\beta +\beta\lambda\left(\frac{1}{p}-\frac{1}{q}\right)$ and $1\leqslant \mu <+\infty$ because $(r,q,p)$ is $\beta_w$-$\mathscr{L}$-admissible, $1<\beta<2$ and $1\leqslant\rho\leqslant r <2$, making sense of the final integration in time. 
\end{proof}
We conclude this subsection by considering the nonlinear $\mathscr{L}$-wave type equation
\begin{align}\label{NLwave}
\begin{split}
\,^{C}\partial_{t}^{\beta}w(t,x)+\mathscr{L}w(t,x)&=F(t,w), \quad t>0,\quad x\in G,\quad 1<\beta<2, \\
w(t,x)|_{_{_{t=0}}}&=w_0(x), \\
\partial_t w(t,x)|_{_{_{t=0}}}&=w_1(x).
\end{split}
\end{align}
In this case the definition of well-posedness is modified in order to take into account the two different data appearing in equation \eqref{NLwave}, so that we formulate it in a product space.  
\begin{defn}
    Let $1<p_0\leqslant2$. We say that the problem \eqref{NLwave} is \textit{locally well-posed} in $L_x^{p_0}(G)\times L_x^{p_0}(G)$ if for any $(w_0^*, w_1^*)\in L_x^{p_0}(G)\times L_x^{p_0}(G)$ there exist a time $T$ and an open ball $B\subset L_x^{p_0}(G)\times L_x^{p_0}(G)$ containing $(w_0^*, w_1^*)$, and a subset $X$ of $C_t^0L_x^{p_0}([0,T]\times G)$, such that for any $(w_0,w_1)\in B$ there exists a strong unique solution $w\in X$ to the integral equation 
   \begin{align*}
    w(t,x)=E_{\beta}(-t^{\beta}\mathscr{L})w_0(x)+&\prescript{RL}{0}I^{1}_{t} E_{\beta}(-t^{\beta}\mathscr{L}) w_1(x) \nonumber \\
    &+  \int_0^t (t-s)^{\beta-1}E_{\beta,\beta}(-(t-s)^\beta\mathscr{L})F(s,w)\, {\rm d}s,
\end{align*}
    and furthermore the map $(w_0,w_1)\mapsto w$ is continuous from $B$ to $X$. 
\end{defn}
Let us fix $1< p_0\leqslant2$, and let $S_{W}^{p_0}([0,T]\times G)$ be defined as the closure of the Schwartz-Bruhat functions under the norm 
\[
\|w\|_{S_{W}^{p_0}([0,T]\times G)} := \sup_{(r,q,p_0) \,\beta_w\text{-}\mathscr{L}\text{-admissible}} \|w\|_{C_t^0L_x^{p_0}([0,T]\times G)}+\|w\|_{L_t^r L_x^q([0,T]\times G)}. 
\] 
We omit the proof of the following theorem since it is essentially the same as the one of Theorem \ref{ThNLH}. For this proof, one needs to use the following inequality, that comes from inequality \eqref{strichWave} (with a non-linear term $F$):   
\begin{align}\label{nonLinHeatEstT-9}
    \|w(t,x)&\|_{S_{W}^{p_0}([0,T]\times G)} \nonumber\\
    \leqslant &\Tilde{C}_{\beta,\lambda,p_0,\rho,T}\left(\|w_0\|_{L_x^{p_0}(G)} + \|w_1\|_{L_x^{p_0}(G)}+\|F(t,w)\|_{L_t^{\rho} L_x^{p_0}([0,T] \times G)}\right),
\end{align}
for any $0<T<+\infty$, some $1\leqslant\rho<2$ and a big enough constant $\Tilde{C}_{\beta,\lambda,p_0,\rho,T}$. This is the constant appearing in the forthcoming theorem, we remark that it is different from the ones appearing for the heat and heat-type equations. 

\begin{thm}
Let $1<p_0\leqslant2$. Suppose that the operator $\mathscr{L}$ satisfies the condition \eqref{trace-condition}. Suppose that there exist $0<T<+\infty$ and $1\leqslant\rho< 2$ such that the nonlinearity $F$ satisfies the following estimate
    \begin{equation*}
        \|F(u) - F(v)\|_{L_t^{\rho} L_x^{p_0}([0,T] \times G)}\leqslant\frac{1}{2\Tilde{C}_{\beta,\lambda,p_0,\rho,T}} \|u -v\|_{S_{W}^{p_0}([0,T]\times G)}
    \end{equation*}
    for all $u,v\in B_{\varepsilon}:= \{u\in S_{W}^{p_0}([0,T]\times G): \|u\|\leqslant \varepsilon\}$, for some $\varepsilon>0$, where the constant $\Tilde{C}_{\beta,\lambda,p_0,\rho,T}$ is from \eqref{nonLinHeatEstT-9}. 
    Then the problem \eqref{NLwave} is locally well-posed in $L_x^{p_0}(G)\times L_x^{p_0}(G)$. 
\end{thm}
\subsection{$\mathscr{L}$-evolutionary integral equation} We conclude this section with a discussion on how one could execute the preceding studies for a more general type of equations. However we will see that the generality does not allow us to give a precise answer. Most likely one would need to perform a different analysis depending on each particular situation. 

\medskip We consider the following integral  equation of scalar type: 
\begin{equation}\label{volterra-e-locallyNo}
w(t,x)=f(t,x)+\int_0^t k(t-s)\mathscr{L}w(s,x)\,{\rm d}s,\quad t\in [0,T],\quad x\in G,
\end{equation}
where $G$ is a separable unimodular locally compact group and $\mathscr{L}$ is a positive left invariant operator acting on $G$. Assuming that $f$ is in a suitable Sobolev space and $f(0,x)=0$ for any $x\in G$, by \cite[Proposition 1.2]{pruss} a mild solution to equation \eqref{volterra-e-locallyNo} is given by 
\[
w(t,x) = \int_0^t S(t-s)g(s,x) \,{\rm d}s,\quad g(s,x):=\partial_s f(s,x). 
\]
Thus using the representation above, let us try to estimate the mixed-type norm of the solution: 
\begin{align*}
    \|w(t,x)\|_{L_t^r L_x^q(\R^+ \times G)} & =  \left\|\int_0^t S(t-s)g(s,x)\,{\rm d}s\right\|_{L_t^r L_x^q([0,T] \times G)} \\
    & = \left\| \left\|\int_0^t S(t-s)g(s,x) \,{\rm d}s\right\|_{L_x^q(G)}\right\|_{L_t^r([0,T])}\\
    &\leqslant \left\|\int_0^t \left\|S(t-s)g(s,x)\right\|_{L_x^q(G)} \,{\rm d}s\right\|_{L_t^r([0,T])}\\
    &\leqslant  C_{\vec{\beta},\lambda,p,q} \left\|\mathfrak{p}(t)*\|g(t, \cdot)\|_{L_x^p(G)}\right\|_{L_t^r([0,T])} \text{ (by Theorem \ref{integral-thm})}\\
\end{align*}
where $\mathfrak{p}(t):= \bigg(\int_0^{t} k(\tau){\rm d}\tau\bigg)^{-\lambda\left(\frac{1}{p}-\frac{1}{q}\right)}.$ At this point we would have to deal with convolution inside of the latter expression. Thus our intention is to proceed utilizing the Young inequality and this brings up the problem of proving finiteness of the following norm
\[
\left\|\mathfrak{p}(t)\right\|_{L_t^\mu([0,T])},\quad \text{where}\quad \frac{1}{\mu}:=1+\frac{1}{r}-\frac{1}{\rho},\quad\text{for some suitable}\quad \rho.
\]
In all previous equations we were dealing with the case $\mathfrak{p}(t)=t^{-\alpha}$ for some $\alpha$, so in order to guarantee the existence of the $L_t^\mu$ norm we just needed to ask for the condition $\alpha\mu<1$ which resulted in the definition of triples. In this general framework is not easy to provide a systematic study, rather we suspect that each different kernel $k$ would require an specific analysis creating different families of triples.  

\section{Acknowledgements}
The authors were supported by the FWO Odysseus 1 grant G.0H94.18N: Analysis and Partial Differential Equations and by the Methusalem programme of the Ghent University Special Research Fund (BOF) (Grant number 01M01021). Michael Ruzhansky is also supported by the EPSRC grant EP/V005529/1 and FWO Senior Research Grants G011522N and G022821N. No new data was collected or generated during the course of this research.


\begin{thebibliography}{00}

\bibitem{RR2020} R. Akylzhanov, M. Ruzhansky. $L^p-L^q$ multipliers on locally compact groups. J. Funct. Anal., 278(3), (2020), \#108324.


\bibitem{BorelFunctional} W. Arveson. A Short Course on Spectral Theory, vol. 209, Springer Science \& Business Media, 2006. 

\bibitem{bao} W. Baoxiang, H. Zhaohui, H. Chengchun, G. Zihua. Harmonic analysis method for nonlinear evolution equations. I. World Scientific Publishing Co. Pte. Ltd., Hackensack, NJ, 2011.

\bibitem{thesis2001} E. Bazhlekova. Fractional Evolution Equations in Banach Spaces. Ph.D. Thesis, Eindhoven University of Technology, 2001.

\bibitem{multi-cmf} E. Bazhlekova. Completely monotone multinomial Mittag-Leffler type functions and diffusion equations with multiple time-derivatives. Fract. Calc. Appl. Anal., 24(1), (2021), 88--111.

\bibitem{apli0} E. Bazhlekova, B. Jin, R. Lazarov, Z. Zhou. An analysis of the Rayleigh-Stokes problem for a generalized second-grade fluid. Numer. Math. 131, 1-31 (2015).


\bibitem{[21]} C. Berg, G. Forst. Potential Theory of Locally Compact Groups, volume 87 of Ergebn. Math. Grenzgeb. Springer Verlag, Berlin, 1975.

\bibitem{von1} B. Blackadar. Operator Algebras: Theory of $C^*$-Algebras and Von Neumann Algebras.  Vol. 122, Springer Science \& Business Media, 2006.

\bibitem{sobolev} H. Brezis. Op\'erateurs maximaux monotones et semi-groupes de contractions dans les espaces de Hilbert. Math. Studies 5, North-Holland, Amsterdam, 1973.

\bibitem{bruhat} F. Bruhat, Distributions sur un groupe localement compact et applications \`{a} l'\'{e}tude des repr\'{e}sentations des groupes $\mathfrak{p}$-adiques. Bull. Soc. Math. France, 89, (1961), 43--75.

\bibitem{intro-p1} F.W. Carroll. Difference properties for continuity and Riemann integrability on locally compact groups. Trans. Amer. Math. Soc. 102(1962), 284--292.

\bibitem{thesis} P.M. Carvalho-Neto. Fractional differential equations: a novel study of local and global solutions in Banach spaces. PhD thesis, Universidade de S\~ao Paulo, S\~ao Carlos, 2013.

\bibitem{Marianna} M. Chatzakou. A note on spectral multipliers on Engel and Cartan groups. Proc. Amer. Math. Soc., 150(5), (2022), 2259--2270.

\bibitem{[50]} Ph. Cl\'ement, J.A. Nohel. Abstract linear and nonlinear Volterra equations preserving positivity. SIAM J. Math. Anal., 10, (1979), 365--388.

\bibitem{Clement} Ph. Cl\'ement, J. A. Nohel. Asymptotic behavior of solutions of nonlinear Volterra equations with completely positive kernels. SIAM J. Math. Anal., 12 (1981), 514--534

\bibitem{intro-nuevo-3} Y. Cornulier, P. de la Harpe. Metric geometry of locally compact groups. EMS Tracts in Mathematics, 25, European Mathematical Society (EMS), Z\"urich, 2016.

\bibitem{von} J. Dixmier. Von Neumann Algebras. North-Holland, Amsterdam, 1981.

\bibitem{intro-p4} G.A. Edgar, J. M. Rosenblatt.
Difference equations over locally compact abelian groups. Trans. Amer. Math. Soc. 253, (1979), 273--289.

\bibitem{apli5} C. Fetecau, M. Jamil, C. Fetecau, D. Vieru. The Rayleigh-Stokes problem for an edge in a generalized ldroyd-B fluid. Z. Angew. Math. Phys. 60(5), (2009), 921--933.

\bibitem{FR16} V. Fischer, M. Ruzhansky. Quantization on nilpotent {L}ie groups. Progress in Mathematics, vol. 314, Birkh\"{a}user/Springer, [Cham], 2016.

\bibitem{intro-nuevo11} A.M. Gleason. On the Structure of locally compact groups. PNAS USA, 35(7), (1949), 384--386.

\bibitem{intro-nuevo1} A.M. Gleason. The structure of locally compact groups. Duke Math. J. 18 (1951), 85--104.

\bibitem{SRR} S. G\'omez Cobos, J.E. Restrepo, M. Ruzhansky. Heat-wave-Schr\"odinger type equations on locally compact groups. arXiv:2302.00721, (2023).

\bibitem{cras} S. G\'omez Cobos, J.E. Restrepo, M. Ruzhansky. $L^p-L^q$ estimates for non-local heat and wave type equations on locally compact groups. C. R. Acad. Sci. Paris, (2024), (to appear).

\bibitem{mittag} R. Gorenflo, A. A. Kilbas, F. Mainardi, S. V. Rogosin. Mittag-Leffler Functions, Related Topics and Applications, 2nd ed. Springer Monographs in Mathematics, Springer, New York, 2020.

\bibitem{functionalcalculus} M. Haase. The Functional Calculus for Sectorial Operators. Oper. Theory Adv.
Appl., 169, Birkh\"auser, 2006.

\bibitem{[35]} A. Hassannezhad, G. Kokarev. Sub-Laplacian eigenvalue bounds on sub-Riemannian manifolds. Ann. Sc. Norm. Super. Pisa Cl. Sci. (5), 16(4), (2016), 1049--1092.

\bibitem{composition} R. Kadison, J. Ringrose. Fundamentals of the Theory of Operator Algebras. Volume I: Elementary Theory. Graduate Studies in Mathematics, 1997.

\bibitem{KeelTao} M. Keel, T. Tao. Endpoint Strichartz estimates. Amer. J. Math., 120(5), (1998), 955--980.

\bibitem{uno2} J. Kemppainen, J. Siljander, V. Vergara, R. Zacher. Decay estimates for time--fractional and other non--local in time subdiffusion equations in $\mathbb{R}^d$. Math. Ann., 366(3), (2016), 941--979.




\bibitem{[198]} J.F.C. Kingman. Regenerative Phenomena. John Wiley and Sons, London, 1972.

\bibitem{intro-nuevo-2} A. W. Knapp. Compact and Locally Compact Groups. In: Advanced Real Analysis. Cornerstones. Birkh\"auser Boston, (2005).

\bibitem{[51]} H. Kosaki.  Non-commutative Lorentz spaces associated with a semi-finite von Neumann algebra and applications.  Proc. Japan Acad. Ser. A Math. Sci.,  57(6), (1981), 303--306.

\bibitem{EstCauchTimeSpac} C. Miao, B. Yuan, B. Zhang. Well-posedness of the {C}auchy problem for the fractional power dissipative equations. Nonlinear Anal., 68(3), (2008), 461--484. 

\bibitem{[46]} F.J. Murray, J. von Neumann. On rings of operators. Ann. of Math. (2), 37(1), (1936), 116–229.

\bibitem{[47]} F.J. Murray, J. von Neumann. On rings of operators II. Trans. Amer. Math. Soc. 41(2), (1937), 208–248.

\bibitem{intro-p3} R.C. Penney, A.L. Rukhin. d'Alembert's functional equation on groups. Proc. Amer. Math. Soc. 77(1), (1979), 73--80.

\bibitem{page 35} I. Podlubny. Fractional Differential Equations. Academic Press, San Diego, 1999.

\bibitem{Pollard} H. Pollard. The completely monotonic character of the Mittag-Leffler function $E_a(-x)$. Bull. Amer. Math. Soc. 54, (1948), 1115--1116.

\bibitem{pruss} J. Pr\"uss. Evolutionary integral equations and applications. Birkh\"auser, Basel, Boston, Berlin, 1993.


\bibitem{RRT} J.E. Restrepo, M. Ruzhansky, B.T. Torebek. Integro-differential diffusion equations on graded Lie groups. Asymptotic Anal., (2024), (to appear).

\bibitem{david} D. Rottensteiner, M. Ruzhansky. Harmonic and anharmonic oscillators on the Heisenberg Group. J. Math. Phys. 63, 111509 (2022).

\bibitem{david2} D. Rottensteiner, M. Ruzhansky. An update on the $L^p-L^q$ norms of spectral multipliers on unimodular Lie groups. Arch. Math., 120 (2023), 507--520.

\bibitem{samko} S.G. Samko, A.A. Kilbas, O.I. Marichev. Fractional integrals and derivatives, translated from the 1987 Russian original, Gordon and Breach, Yverdon, 1993.

\bibitem{apli28} F. Shen, W. Tan, Y. Zhao, T. Masuoka. The Rayleigh-Stokes problem for a heated generalized second grade fluid with fractional derivative model. Nonlinear Anal. Real World Appl. 7(5), (2006), 1072--1080.

\bibitem{intro-p2} E.V. Shulman. Group representations and stability of functional equations. J. London Math. Soc. (2)54, (1996), 111--120.

\bibitem{sonine} N. Sonine. Sur la g\'en\'eralisation d’une formule d’Abel. Acta Math., 4, (1884), 171--176.

\bibitem{Stri}  R.S. Strichartz. Restrictions of Fourier transforms to quadratic surfaces and decay of solutions of wave equations. Duke Math. J., 44(3), (1977), 705--714.

\bibitem{intro-nuevo-4} M. Stroppel. Locally compact groups. EMS Textbooks in Mathematics, European Mathematical Society (EMS), Z\"urich, 2006.




\bibitem{tao} T. Tao. Nonlinear dispersive equations: local and global analysis. CBMS regional conference series in mathematics, 2006.

\bibitem{von2} M. Terp. $L^p$ Spaces Associated with Von Neumann Algebras. Copenhagen University, 1981.

\bibitem{pacific} F. Thierry, H. Kosaki. Generalized $s$-numbers of $\tau$-measurable operators. Pacific J. Math. 123(2), (1986), 269--300.





\bibitem{Vergara1} V. Vergara, R. Zacher. Optimal decay estimates for time-fractional and other nonlocal subdiffusion equations via energy methods. SIAM J. Math. Anal., 47(1), (2015), 210--239.

\bibitem{[30]} V.S. Vladimirov, I. Volovich, E. Zelenov. $p$-adic analysis and mathematical physics. Series on Soviet and East European mathematics, V1. World Scientific, 1994.

\end{thebibliography}
\end{document}